\newif\ifcomment
\commenttrue

\documentclass[reqno,twoside,11pt]{amsart}
\usepackage{cite}
\usepackage{amsmath}
\usepackage{amsfonts}
\usepackage{amssymb}
\usepackage{epsfig}
\usepackage{verbatim}
\usepackage{color}

\usepackage{mathpazo}
\usepackage{mathabx}

\DeclareFontFamily{OT1}{cmss}{} \DeclareFontShape{OT1}{cmss}{m}{n} {<5> <6> <7> <8> <9> <10> <11> <12> <13> <14.4> cmss10}{}
\DeclareMathAlphabet{\cmss}{OT1}{cmss}{m}{n}

\DeclareFontFamily{OT1}{fraktura}{}
\DeclareFontShape{OT1}{fraktura}{m}{n} {<5> <6> <7> <8> <9> <10> <11> <12> <13> <14.4> [1.1] eufm10}{}
\DeclareMathAlphabet{\fraktura}{OT1}{fraktura}{m}{n}

\newtheoremstyle{thm}{1.5ex}{1.5ex}{\itshape\rmfamily}{} {\bfseries\rmfamily}{}{2ex}{}

\newtheoremstyle{def}{1.5ex}{1.5ex}{\rmfamily\sl}{} {\bfseries\rmfamily}{}{2ex}{}

\newtheoremstyle{rem}{1.3ex}{1.3ex}{\rmfamily}{} {\bfseries\rmfamily}{}{2ex}{}

\newtheoremstyle{ass}{1.5ex}{1.5ex}{\rmfamily\sl}{} {\bfseries\rmfamily}{}{2ex}{}

\newenvironment{proofsect}[1] {\vskip0.1cm\noindent{\rmfamily\itshape#1.}}{\qed\vspace{0.15cm}}

\theoremstyle{thm}
\newtheorem{theorem}{Theorem}[section]
\newtheorem{lemma}[theorem]{Lemma}
\newtheorem{proposition}[theorem]{Proposition}

\newtheorem*{Main Theorem}{Main Theorem.}
\newtheorem{corollary}[theorem]{Corollary}

\newtheorem{definition}[theorem]{Definition}

\theoremstyle{rem}
\newtheorem{remark}[theorem]{{Remark}}

\numberwithin{equation}{section}

\renewcommand{\section}{\secdef\sct\sect}
\newcommand{\sct}[2][default]{\refstepcounter{section}
\addcontentsline{toc}{section}
{{\tocsection {}{\thesection}{\!\!\!\!#1\dotfill}}{}}
\vspace{0.7cm}
\centerline{ 
\scshape\arabic{section}.\ #1} \nopagebreak \vspace{0.2cm}}
\newcommand{\sect}[1]{
\vspace{0.4cm} \centerline{\large\scshape\rmfamily #1}
\vspace{0.2cm}}

\renewcommand{\subsection}{\secdef\subsct\sbsect}
\newcommand{\subsct}[2][default]{\refstepcounter{subsection}
\addcontentsline{toc}{subsection}
{{\tocsection{\!\!}{\hspace{1.2em}\thesubsection}{\!\!\!\!#1\dotfill}}{}}
\nopagebreak\vspace{0.45\baselineskip} {\flushleft\bf
\thesection.\arabic{subsection}~\bf #1.~}
\\*[3mm]\noindent
\nopagebreak}
\newcommand{\sbsect}[1]{
\vspace{0.1cm}\noindent
\textbf{#1.~}\vspace{0.1cm}}

\renewcommand{\subsubsection}{%
\secdef \subsubsect\sbsbsect}
\newcommand{\subsubsect}[2][default]{%
\refstepcounter{subsubsection} 
\addcontentsline{toc}{subsubsection}{{\tocsection{\!\!}
{\hspace{3.05em}\thesubsubsection}{\!\!\!\!#1\dotfill}}{}}
\nopagebreak
\vspace{0.15\baselineskip} \nopagebreak {\flushleft\rmfamily
\itshape\arabic{section}.\arabic{subsection}.\arabic{subsubsection}
\ \rmfamily #1\/.}\ }
\newcommand{\sbsbsect}[1]{\vspace{0.1cm}\noindent
\rmfamily \itshape
\arabic{section}.\arabic{subsection}.\arabic{subsubsection} \
\sffamily #1\/.\ }

\renewcommand{\caption}[1]{%
\vglue0.5cm
\refstepcounter{figure}
\begin{center}
\begin{minipage}[c]{0.8\textwidth}\small {\sc Fig.~\thefigure\ }#1\end{minipage}
\end{center}
}

\newcommand{\dist}{\operatorname{dist}}
\newcommand{\supp}{\operatorname{supp}}
\newcommand{\diam}{\operatorname{diam}}

\newcommand{\textd}{\text{\rm d}\mkern0.5mu}

\newcommand{\texte}{\text{\rm  e}\mkern0.7mu}

\newcommand{\Var}{\text{\rm Var}}

\newcommand{\HH}{\mathcal H}
\newcommand{\II}{\mathcal I}

\newcommand{\MM}{\mathcal M}

\newcommand{\WW}{\mathcal W}

\newcommand{\E}{\mathbb E}
\newcommand{\bbE}{\E}

\newcommand{\N}{\mathbb N}

\newcommand{\BbbP}{\mathbb P}
\newcommand{\bbP}{\BbbP}
\newcommand{\Q}{\mathbb Q}
\newcommand{\R}{\mathbb R}
\newcommand{\bbR}{\R}

\newcommand{\Z}{\mathbb Z}
\newcommand{\bbZ}{\Z}

\newcommand{\twoeqref}[2]{(\ref{#1}--\ref{#2})}

\newcommand{\cc}{{\text{\rm c}}}

\newcommand{\frakb}{\fraktura b}

\def\myffrac#1#2 in #3{\raise 2.6pt\hbox{$#3 #1$}\mkern-1.5mu\raise 0.8pt\hbox{$#3/$}\mkern-1.1mu\lower 1.5pt\hbox{$#3 #2$}}
\newcommand{\ffrac}[2]{\mathchoice%
	{\myffrac{#1}{#2} in \scriptstyle}
	{\myffrac{#1}{#2} in \scriptstyle}
	{\myffrac{#1}{#2} in \scriptscriptstyle}
	{\myffrac{#1}{#2} in \scriptscriptstyle}
}

\newcommand{\wh}{\widehat}
\newcommand{\wt}{\widetilde}
\newcommand{\ol}{\overline}

\newcommand{\laweq}{\,\overset{\text{\rm law}}=\,}

\newcommand{\leb}{{\rm Leb}}

\newcommand{\Lawarrow}{{\,\overset{\text{\rm law}}\longrightarrow\,}}

\newcommand{\rmd}{\textd}
\newcommand{\rme}{\texte}

\newcommand{\ep}{\epsilon}

\newcommand{\fb}{\frakb}

\begin{document}

\vglue-0mm

\title[Support sets of critical LQG\hfill]{
 On support sets of the critical Liouville Quantum Gravity}
\author[\hfill M.~Biskup, S. Gufler, O. Louidor]
{Marek~Biskup$^{1}$,\, Stephan Gufler$^2$ and Oren Louidor$^3$}
\thanks{\hglue-4.5mm\fontsize{9.6}{9.6}\selectfont\copyright\,\textrm{2024}\ \ \textrm{M.~Biskup, S. Gufler, O. Louidor. Reproduction for non-commercial purposes is permitted.\vspace{2mm}}}
\maketitle

\vspace{-5mm}
\centerline{\textit{$^1$
Department of Mathematics, UCLA, Los Angeles, California, USA}}
\centerline{\textit{$^2$
Institut f\"ur Mathematik, Goethe-Universit\"at, Frankfurt am Main, Germany}}
\centerline{\textit{$^3$
Faculty of Industrial Engineering and Management, Technion, Haifa, Israel}}
\smallskip


\vskip4mm
\begin{quote}
\footnotesize \textbf{Abstract:}
We study fractal properties of support sets of the critical Liouville Quantum Gravity (cLQG) associated with the Gaussian Free Field in planar domains. Specifically, we completely characterize the gauge functions~$\phi$ (subject to mild monotonicity conditions) for which  the cLQG admits a support set of finite $\phi$-Hausdorff measure. As a corollary, we settle the conjecture that the cLQG is supported on a set of vanishing Hausdorff dimension. Our proofs are based on the fact that the cLQG describes the near-critical level sets of Discrete Gaussian Free Field.
\end{quote}


\section{Introduction and results}
\noindent
This note is concerned with fractal properties of the critical Liouville quantum gravity.  This is a particular member of a one-parameter family of random measures  that, in a planar domain~$D\subseteq\R^2$, give a rigorous meaning to the formal expression 
\begin{equation}
\label{E:1.1}
1_D(x)\texte^{\beta h(x)}\leb(\textd x),
\end{equation}
where~$h$ is a sample of the Continuum Gaussian Free Field (CGFF) in~$D$ with Dirichlet boundary conditions and $\leb$ denotes the Lebesgue measure on~$\R^2$.  The expression is formal because the CGFF exists only in the sense of distributions. 

The properties, and our means of control, of the Liouvile quantum gravity (LQG) depend on  the parameter~$\beta$  and its relation to a critical value~$\beta_\cc$ that is determined by the normalization of~$h$.
In the subcritical cases, $0<\beta<\beta_\cc$, the existence and uniqueness goes back to Kahane's theory of Multiplicative Chaos~\cite{Kahane} with an alternative construction based on CGFF circle averages (Duplantier and Sheffield~\cite{DS})  and an intrinsic characterization put forward by~Shamov~\cite{Shamov}.  The critical case, $\beta=\beta_\cc$, requires a subtler approach with existence settled in Duplantier, Rhodes, Sheffield and Vargas~\cite{DRSV1,DRSV2} and uniqueness in Junnila and Saksman~\cite{JS} and Powell~\cite{Powell}. The supercritical LQGs ($\beta>\beta_\cc$) turn out to be purely-atomic measures derived from the critical case (Rhodes and Vargas~\cite{RV-review}, Biskup and Louidor~\cite{BL3}).

The subcritical LQGs are easier to handle and quite a bit is known about them. For instance, while charging all non-empty open subsets of~$D$ with probability one, these measures are supported on a fractal set. Indeed, the combined results of Kahane~\cite{Kahane} and Hu, Peres and Miller~\cite{HMP10} show that a.e.\ sample of the LQG with parameter~$\beta$ is carried by a set of Hausdorff dimension $2-2(\beta/\beta_\cc)^2$.
 For a carrier (i.e., a supporting set) with this property one can choose the set of CGFF thick points; i.e., points with a particular rate of growth of circle averages of~CGFF.  As  shown by two of the present authors~\cite{BL4}, the connection with the GFF ``thick points'' exists also for the lattice version of the GFF, called the Discrete Gaussian Free Field (DGFF). 
 
 Throughout this work  we focus on the critical LQG (i.e.,~$\beta=\beta_\cc$) henceforth abbreviated as cLQG and denoted, in a domain~$D$, by~$Z^D$. 
 Also this measure is closely connected to the DGFF; indeed, it serves as the random intensity measure for a Cox-process limit of the extremal process associated with the DGFF (Biskup and Louidor~\cite{BL1,BL2,BL3}). Apart from that, the cLQG arises as a weak limit of the subcritical LQGs as~$\beta\uparrow\beta_\cc$ (Aru, Powell and Sep\'ulveda~\cite{APS}),  in an approach based on ``mating of trees'' (Duplantier, Sheffield and Miller~\cite{DSM21}, Aru, Holden, Powell and Sun~\cite{AHPS}) or via a coupling to the conformal loop ensemble (Ang and Gwynne~\cite{Ang-Gwynne}).  

Extrapolating the above formula for the subcritical LQG suggests that the cLQG is supported on a set of vanishing Hausdorff dimension. This agrees with the findings of Barral, Kupiainen, Nikula, Saksman and Webb~\cite[Corollary 24]{BKNSW} who proved existence of carriers of vanishing Hausdorff dimension for critical Multiplicative Chaos associated with a class of exactly scale-invariant log-correlated Gaussian fields on~$\R$.  Our main goal  is to establish the corresponding claim for the cLQG.

\subsection{Main results}
We will actually take a more refined approach and consider Hausdorff measures for general (not just power-based) gauge functions. Recall that, given an increasing continuous $\phi\colon[0,\infty)\to[0,\infty)$ with $\phi(0)=0$, the Hausdorff (outer) measure of a set $A\subseteq\R^2$ with respect to the gauge function $\phi$ is given by
\begin{equation}
\mathcal H^{\phi}(A):=\sup_{\delta>0}\inf\left\{
\sum_{i\in\N} \phi\left(\diam \, O_i\right)\colon A\subseteq \bigcup_{i\in\N} O_i\,\wedge\, \sup_{i\in\N} \diam\, O_i\leq\delta\right\} ,
\end{equation}
where $\diam(\cdot)$ is the Euclidean diameter and the infimum is over all collections $\{O_i\}_{i\in\N}$ of non-empty bounded open subsets of~$\R^2$. For the particular choice $\phi_s(r):=r^s$,
\begin{equation}
\dim_{\rm H}(A):=\inf\bigl\{s>0\colon\HH^{\phi_s}(A)=0\bigr\}
\end{equation}
defines the Hausdorff dimension of~$A$. 

Let~$\mathfrak D$ be the set of all bounded open subsets of~$\R^2$  whose topological boundary~$\partial D$  has a finite number of connected components each of which has positive Euclidean diameter. 
 The principal result of our note is then:

\begin{theorem}[Hausdorff measure of  cLQG-carriers]
\label{thm-B}
For~$D\in\mathfrak D$, let $Z^D$ be the cLQG measure on~$D$ and let $\psi\colon(0,\infty)\to(0, \infty)$ be a bounded continuous  decreasing  function  such that~$\lim_{t\to\infty}\sqrt{1+t}\,\psi(t)=\infty$.  
Denoting 
\begin{equation}
\label{E:2.15}
\II_\psi:=\int_{1}^{\infty}\frac{\psi(t)}{t} \,\textd t,\qquad \gamma(t):=\sqrt{1+t\,}\,\psi(t),\qquad
\phi(r):=\int_{\log(1/r)}^\infty \rme^{-\gamma(t)}\textd t,
\end{equation} 
the following dichotomy then holds for a.e.\ sample of~$Z^D$:
\settowidth{\leftmargini}{(1111)}
\begin{enumerate}
\item
If $\II_\psi<\infty$, then $\mathcal H^{\phi}(A) =\infty$ holds for all Borel sets $A\subseteq D$ with $Z^D(A)>0$.
\item
If $\II_\psi=\infty$, then there   exists a Borel set  $A\subseteq D$ with $Z^D(A^\cc)=0$ and $\mathcal H^{\phi}(A)=0$.
\end{enumerate} 

\end{theorem}

Our result demonstrates a delicate fractal nature of cLQG carriers. Indeed, the  convergence/divergence ``threshold'' for~$\II_\psi$ is marked by functions $\phi$ that decay to zero  exponentially in root-log,  i.e., functions  of the form 
\begin{equation}
\label{E:1.6u}
\phi(r)=\exp\bigl\{-(\log(1/r))^{1/2+o(1)}\bigr\},
\end{equation}
 where  the exact nature of the $o(1)$-term (that vanishes as~$r\downarrow0$)  decides  the convergence/divergence of~$\II_\psi$.  Since \eqref{E:1.6u} vanishes (as~$r\downarrow0$) slower than any power, we get:

\begin{corollary}
\label{cor-1.2}
A.e.\ sample of~$Z^D$ admits a carrier of vanishing Hausdorff dimension.
\end{corollary}

Although  the role of \eqref{E:1.6u} as a threshold function  critical  for the existence of carriers of vanishing Hausdorff measure was recognized already in the aforementioned work~\cite{BKNSW} (see, e.g., Theorem~4 therein),  this came  without a full characterization and/or a dichotomy of the kind presented in Theorem~\ref{thm-B}.

\subsection{ A comment and  extensions}
 The threshold between convergence and divergence of~$\II_\psi$ occurs in the class of polylogarithmically decaying~$\psi$; namely,
\begin{equation}
\psi(t):=[\log(1+t)]^{-\theta}
\end{equation}
as~$\theta$ crosses from~$\theta>1$ to~$0<\theta\le1$. This shows that the condition~$\gamma(t)\to\infty$ as~$t\to\infty$ that we impose on~$\psi$ in the statement is rather harmless. (The main reason for imposing this is technical convenience in the proof of Lemma~\ref{lemma-Z-sup} where this condition permits us to avoid dealing with subsequences.)

We proceed with remarks on possible directions of further study. 

\smallskip\noindent
\textsl{Universality: } A natural question to ask is to what other log-correlated random fields do our conclusions apply. For this we need to note that our proof relies on a close link between the growth rate of normalized disc volumes,
\begin{equation}
\label{E:1.6i}
r\mapsto \frac{Z^D(B(\wh X,r))}{\phi(r)},\quad \text{as }r\downarrow0,
\end{equation}
and the function
\begin{equation}
\label{E:1.7i}
t\mapsto Y_t-\gamma(t),\quad \text{as }t\to\infty,
\end{equation}
for  $Y$ denoting  the three dimensional Bessel process; see Proposition~\ref{prop-Z-sup}. Here~$B(x,r)$ is the Euclidean disc of radius~$r$ centered at~$x$ and~$\wh X$ is sampled from a suitably normalized~$Z^D$. The functions~$\phi$ and~$\gamma$ are both related to the same~$\psi$ as in~\eqref{E:2.15}.

The Bessel process itself arises via conditioning from a Brownian motion that describes the fluctuations of CGFF circle averages~\cite{DRSV1,DRSV2} or (as in our proof) a  ``spine''  random walk that governs the so called concentric decomposition of the DGFF~\cite{BL3,Ballot} centered at a given point. Roughly speaking, the conditioning appears because, in order for the field to be maximal (or near maximal) at the centering point, an ``entropic repulsion'' effect forces the Brownian motion to stay above a polylogarithmic curve; see~\cite{CHLsupp}. A Doob $h$-transform translates this into a particular drift thus turning the Brownian motion into a three-dimensional Bessel process.

The argument we just stated suggests that the conclusion of Theorem~\ref{thm-B} should be \emph{universal} for a class of log-correlated processes to which the underlying  ``spine''  Brow\-nian motion representation applies.  We expect the same even for the square-root local time of a random walk on a regular tree (whose extrema were studied by Abe~\cite{A18} and two of the present authors~\cite{BL5}), despite the fact that there the associated ``backbone'' process is a zero-dimensional Bessel process. This is because this process is still close to Brownian motion once sufficiently large, which is the regime relevant here.

\smallskip\noindent
\textsl{Refined fractal structure: }
Our results are directed towards the carriers of~$Z^D$; i.e., sets that support the random measure, and they capture  the rate of growth of $r\mapsto Z^D(B(x,r))$  near its typical points. As a next step, we may ask about the sets of points where~$Z^D$ behaves in a non-typical way. 

For instance, given~$\phi$ such that~$\II_\psi=\infty$, we know that  $Z^D(B(x,r))/\phi(r)\to\infty$ as~$r\downarrow0$ at $Z^D$-a.e.\ $x\in D$;  see Lemma~\ref{lemma-Z-sup}. The set
\begin{equation}
\biggl\{x\in D\colon\,\limsup_{r\downarrow0}\,\frac{Z^D(B(x,r))}{\phi(r)}<\infty\biggr\}.
\end{equation}
is thus $Z^D$-null which means that, for each~$\epsilon>0$, it can be covered by a countable collection of balls $\{B_k\}_{k\ge1}$ such that $\sum_{k\ge1}Z^D(B_i)<\epsilon$. The question is whether there exists a non-trivial exponent~$s>0$ for which $\sum_{k\ge1}Z^D(B_i)^s$ remains uniformly positive for any such cover. 

Similarly, if in turn~$\II_\psi<\infty$, Lemma~\ref{lemma-Z-sup} tells us that  $Z^D(B(x,r))/\phi(r)\to0$ as~$r\downarrow0$ at~$Z^D$-a.e.\ $x\in D$,  but then we may ask the above question about the set
\begin{equation}
\biggl\{x\in D\colon\,\liminf_{r\downarrow0}\,\frac{Z^D(B(x,r))}{\phi(r)}>0\biggr\}.
\end{equation}
Understanding these sets in relation to~$\phi$ near the critical scaling \eqref{E:1.6u} might reveal more refined fractal properties of $Z^D$-measure that would be of independent interest.


\subsection{ Main steps and outline}
 The proof of our main results comes in three parts. The first part is a relatively straightforward reduction of the existence/absence of carriers of non-trivial Hausdorff measure to criteria for gauge functions~$\phi$ for which the normalized disc volumes \eqref{E:1.6i} are either unbounded or tend to zero as~$r\downarrow0$ (which, as it turns out, are the only two possibilities that can occur). This reduces the proof to two inequalities (stated in Proposition~\ref{prop-Z-sup}) that relate the growth/decay of \eqref{E:1.6i} to the asymptotic behavior of the process \eqref{E:1.7i}, with~$\gamma$ tied to~$\phi$ via \eqref{E:2.15}.

The second part of the proof is a reduction of the inequalities in Proposition~\ref{prop-Z-sup} to more elementary statements. Here we represent the $Z$-measure as the $N\to\infty$ limit of its counterpart~$Z_N$ defined from the Discrete Gaussian Free Field~$h$ (DGFF) on a scaled-up version~$D_N$ of~$D$ on the square lattice~$\Z^2$. The advantage of this representation is that we can now condition~$h$ on a value at a point (while pointwise values of the continuum GFF are meaningless). This is relevant because the computations are thus reduced to control of the asymptotic behavior of~$Z_N$ on boxes centered at a point~$x_0$ where the field is conditioned on a ``near maximal'' value; i.e., a value order-$\sqrt{\log N}$ below the typical scale of the maximum of DGFF on~$D_N$. Here we benefit from our recent description of the asymptotic distribution of such ``near-maximal'' level sets~\cite{BGL}.

Under the stated conditioning we can invoke the so called concentric decomposition of the DGFF to describe the $Z_N$-measure on concentric annuli centered at~$x_0$ by way of a random walk. As noted earlier, the main point is to control an entropic repulsion effect which, effectively, makes this walk behave as a $3$-dimensional Bessel process. Here we rely heavily on Ballot Theorems in the form studied by two of us in~\cite{Ballot}. The proofs are thus reduced to two lemmas (namely, Lemma~\ref{lemma-3.9} and Lemma~\ref{lemma-3.10}) that encapsulate the third and technically hardest part of the argument.

The remainder of this paper is organized as follows: In Section~\ref{sec-2} we prove the  main  results conditional on the link between \eqref{E:1.6i} and \eqref{E:1.7i}. This link is proved in Section~\ref{sec-4} with the help of the technical lemmas and the concentric decomposition that is introduced and explained in Section~\ref{sec-3}. The technical lemmas are proved in Section~\ref{sec-5}.

\section{ Reduction to  a key proposition}
\label{sec-2}\noindent
The proof of Theorem~\ref{thm-B} comes in several layers. Here we present the ``outer'' layer where the proof  is  constructed conditional on a technical proposition. The proof of this proposition constitutes the remainder of the paper.

\subsection{ Key proposition}
We start with  some notation and definitions. Observe that, given a domain $D\in\mathfrak D$,  the  cLQG in~$D$ is a random variable~$Z^D$ taking values in the space $\MM(\R^2)$ of Radon measures on~$\R^2$ endowed with the  $\sigma$-algebra of  Borel sets   associated with   the topology of vague convergence. This topology is generated by sets $\{\mu\in\MM(\R^2)\colon \int f\textd\mu\in O\}$  indexed by functions  $f\in C_\cc(\R^2)$ and  open sets  $O\subseteq\R^2$.   We will also consider the product space $\R^2\times\MM(\R^d)$ which we endow with the product topology and the associated product $\sigma$-algebra. Standard arguments from measure theory give: 


\begin{lemma}
\label{lemma-2.1}
 For each bounded Borel set~$B\subseteq\R^2$, the function~$(x,\mu)\mapsto\mu(x+B)$ is Borel measurable as a map~$\R^2\times\MM(\R^2)\to\R$.  Moreover,  let $\MM_\star(D)$ be the subset of $\MM(\R^2)$ containing all the measures that have positive and finite total mass  and are  concentrated on~$D$.
Then $\MM_\star(D)$ is,  for each~$D\in\mathfrak D$, a Borel subset of $\MM(\R^2)$  with~$P(Z^D\in\MM_\star(D))=1$.\end{lemma}

\begin{proofsect}{Proof}
Note that $x\mapsto f(x+\cdot)$ is continuous in the supremum norm for all~$f\in C_\cc(\R^2)$. 
Thanks to the use of the vague topology on~$\MM(\R^2)$ and continuity of $f\mapsto\int f\textd\mu$ in the supremum norm for~$f$ supported in a given bounded set, $(x,\mu)\mapsto\int f(x+\cdot)\textd\mu$ is continuous and thus measurable for each~$f\in C_\cc(\R^2)$. Using this for~$f$ replaced by $ f_n (x):=\min\{1,n\dist(x,\R^2\smallsetminus B)\}$ which, we note, obeys $ f_n  \uparrow 1_{\text{int}(B)}$ as $n\to\infty$, the Bounded Convergence Theorem yields measurability of $(x,\mu)\mapsto\mu(x+B)$ for each bounded open~$B\subseteq\R^2$. Dynkin's $\pi$-$\lambda$ Theorem then extends this to all bounded Borel sets $B\subseteq\R^2$, proving the first part of the claim.

For the second part of the claim we note that, since~$D$ is bounded, $\mu(D)<\infty$ holds for all~$\mu\in\MM(\R^2)$. Then observe that~$\MM_\star(D)$ can be written as the intersection of two sets: first, the intersection of  the  sets $\{\mu\in\MM(\R^d)\colon \int f\textd\mu=0\}$  with  $f\in C_\cc(\R^2)$  subject to  $\supp(f)\cap D=\emptyset$, which is closed in~$\MM(\R^d)$, and, second, the union of  the  sets $\{\mu\in\MM(\R^d)\colon \int f\textd\mu\ne0\}$  with  $f\in C_\cc(\R^2)$  subject to  $\supp(f)\subseteq D$, which is open in~$\MM(\R^2)$. In particular,~$\MM_\star(D)$ is a Borel subset of~$\MM(D)$. That $P(Z^D\in\MM_\star(D))=1$ is known from, e.g.,~\cite[Theorem~2.1]{BL2}.
\end{proofsect}

The fact that~$Z^D(D)>0$ a.s.\ allows us to consider the normalized version
\begin{equation}
\wh Z^D(A):=\frac{Z^D(A)}{Z^D(D)}
\end{equation}
of~$Z^D$ and write~$\wh X$ for a sample from~$\wh Z^D$, conditional on~$Z^D$. We take this to mean that the joint law of~$\wh X$ and~$Z^D$ is a probability measure on~$\R^2\times\MM(\R^2)$ endowed with the product Borel $\sigma$-algebra, such that for all Borel sets~$A\subseteq\R^2$ and $B\subseteq\MM(\R^2)$,
\begin{equation}
\label{E:2.1}
P\bigl(\wh X\in A,\,Z^D\in B\bigr)=E\bigl(\,\wh Z^D(A) 1_{\{Z^D\in B\}}\bigr),
\end{equation}
where~$E$ is the expectation with respect to the law of~$Z^D$. The above product space will be assumed as our underlying probability space in the sequel. 

\begin{remark}
The measure on the right of \eqref{E:2.1} is akin to the \emph{Peyri\`ere rooted measure} used frequently in the studies of  Gaussian Multiplicative Chaos. The difference is that for the Peyri\`ere measure the size biasing is done using~$Z^D$ and not its normalized counterpart. In particular, since $Z^D$ fails to have the first moment, the Peyri\`ere measure has infinite mass. This causes problems that can be overcome by suitable truncations (e.g., as in Duplantier, Rhodes, Sheffield and Vargas~\cite{DRSV1}). We opt to work with the normalized measure instead as that fits better our approach.
\end{remark}

 A key tool of our proofs is a tight control of the decay rate of $Z^D$-mass on balls of small radii centered at the point~$\wh X$ sampled from the size-biased law \eqref{E:2.1}:

\begin{proposition}[Key proposition]
\label{prop-Z-sup}
Let~$\gamma\colon(0,\infty)\to(0,\infty)$ be a continuous function with $\psi(t):=\gamma(t)/\sqrt{1+t\,}$  decreasing for large~$t$ and such that, for some~$\eta>0$, we have $t^{-\eta}\gamma(t)\to\infty$ as~$t\to\infty$. For each~$\epsilon>0$ there is~$c\in(0,\infty)$ such that
\begin{equation}
\label{E:2.3b}
P\biggl(\limsup_{r\downarrow0}\frac{Z^D(B(\wh X,r))}{\int_{\log(1/r)}^\infty\rme^{-4\gamma(t)}\textd t}\le 1\biggr)\le 
\epsilon + cP\Bigl(\liminf_{t\to\infty}\,[Y_t-\gamma(t)\bigr]\ge0\Bigr)
\end{equation}
and
\begin{equation}
\label{E:2.4}
P\biggl(\limsup_{r\downarrow0}\frac{Z^D(B(\wh X,r))}{\int_{\log(1/r)}^\infty\rme^{-\gamma(t)/2}\textd t}\ge1\biggr)\le 
\epsilon+ c
P\Bigl(\liminf_{t\to\infty}\,[Y_t-2\gamma(t)]\le0\Bigr)
\end{equation}
hold with $\{Y_t\}_{t\ge0}$ denoting the 3-dimensional Bessel process started at~$0$.
\end{proposition}

 For definiteness, we recall that the $3$-dimensional Bessel process is a solution to the SDE $\textd Y_t = \textd W_t+Y_t^{-1}\textd t$, where~$W$ is a standard Brownian motion.  (Alternatively,~$Y$ is the modulus of the three-dimensional standard Brownian motion.)  
The numerical factors  $4$,  $2$ and $1/2$ in \twoeqref{E:2.3b}{E:2.4} are included for convenience of the proofs; they can be moved from one side to the other by rescaling~$\gamma$. Rescaling arguments are also key for bringing the conclusion of  Proposition~\ref{prop-Z-sup} to that of Theorem~\ref{thm-B}.

 Note that the statement of \twoeqref{E:2.3b}{E:2.4}  tacitly assumes that the \emph{limes superior} on the left-hand side is a random variable. This is addressed in: 

\begin{lemma}
\label{lemma-2.6}
Consider the product space $\R^2\times\MM(\R^2)$ as discussed above and let~$\MM_\star(D)$ be as in Lemma~\ref{lemma-2.1}. Then the function $f_\phi\colon\R^2\times\MM(\R^2)\to[0,\infty]$ defined by
\begin{equation}
\label{E:2.17}
f_\phi(x,\mu):=1_{\MM_\star(D)}(\mu)\limsup_{r\downarrow0}\frac{\mu(B(x,r))}{\phi(r)}
\end{equation}
is Borel measurable for  every  continuous~$\phi\colon(0,\infty)\to(0,\infty)$. In particular, assuming $(\wh X,Z^D)$ are realized as coordinate projections on $\R^2\times\MM(\R^2)$, $\limsup_{r\downarrow0}Z^D(B(\wh X,r))/\phi(r)$
is (equivalent to) a random variable on the underlying probability space~$\R^2\times\MM(\R^2)$.
\end{lemma}

\begin{proofsect}{Proof}
 By Lemma~\ref{lemma-2.1}, $(x,\mu)\mapsto \mu(B(x,r))$ is Borel measurable for each~$r\ge0$. The monotonicity of~$r\mapsto \mu(B(x,r))$ along with continuity and positivity of~$\phi(r)$ permit writing the  \emph{limes superior} in \eqref{E:2.17} as
\begin{equation}
\inf_{n\ge1}\sup_{r\in\Q\cap(0,1/n)}\frac{\mu(B(x,r))}{\phi(r)}
\end{equation}
which is thus Borel measurable. Since $\MM_\star(D)$ is a Borel set, this proves that~$f_\phi$ from \eqref{E:2.17} is a Borel function and $f_\phi(\wh X,Z^D)$ is a random variable. As $P(Z^D\in\MM_\star(D))=1$, the second part of the statement follows as well. 
\end{proofsect}

\subsection{ Proof of main results from Proposition~\ref{prop-Z-sup}}
 We will now show how Proposition~\ref{prop-Z-sup} implies our main results. First  we recast \twoeqref{E:2.3b}{E:2.4}  as an a.s.\ limit statement:

\begin{lemma}
\label{lemma-Z-sup}
For $\phi$, $\psi$  and  $\II_\psi$ as in the assumptions of Theorem~\ref{thm-B}, the following holds for a.e.-realization of~$Z^D$: For a.e.\ sample of~$\wh X$  from~$\wh Z^D$, 
\begin{equation}
\label{E:2.1a}
\limsup_{r\downarrow0}\frac{Z^D(B(\wh X,r))}{\phi(r)}=\begin{cases}
\infty,\qquad&\text{if }\II_\psi=\infty,
\\
0,\qquad&\text{if }\II_\psi<\infty.
\end{cases}
\end{equation}
\end{lemma}

\begin{proofsect}{Proof from Proposition~\ref{prop-Z-sup}}
The proof is based on an integral characterization of the minimal asymptotic growth of the $d$-dimensional Bessel processes. These have been studied by Motoo~\cite{Motoo} with the corresponding results on the growth rate of the simple random walk in $d\ge3$ derived already by Dvoretzky and Erd\H{o}s~\cite{Dvoretzky-Erdos}. For the $3$-dimensional Bessel process~$Y$ and any bounded $\psi\colon(0,\infty)\to(0,\infty)$ with $t\mapsto\psi(t)$  decreasing for~$t$ sufficiently large, the second Theorem on page~27 of Motoo's paper~\cite{Motoo} (with ``$k-2/2$'' in the exponent  corrected to ``$k-2$'')   gives
\begin{equation}
\label{E:2.9a}
P\Bigl(\liminf_{t\to\infty}\,\bigl[Y_t-\sqrt{1+t\,}\,\psi(t)\bigr]\le0\Bigr)=\begin{cases}1,\qquad&\text{ if }\II_{\psi}=\infty,\\
0,\qquad&\text{ if }\II_{\psi}<\infty.
\end{cases}
\end{equation}
Here we simplified  Motoo's  integral criterion using the fact that~$\psi$ is bounded.

 As is  readily checked,  rescaling  of~$\psi$ by a positive number does not affect the convergence of~$\II_{\psi}$.  By assumption,  
\begin{equation}
\label{E:2.11a}
\gamma(t):=\sqrt{1+t}\,\psi(t)\,\,\underset{t\to\infty}\longrightarrow\,\,\infty,
\end{equation}
the dichotomy \eqref{E:2.9a} therefore gives
\begin{equation}
\label{E:2.12a}
\II_{\psi}=\infty\quad\Rightarrow\quad 
P\Bigl(\liminf_{t\to\infty}\bigl[Y_t-\gamma(t)\bigr]\ge0\Bigr)=0
\end{equation}
because the \emph{limes inferior} equals negative infinity a.s., and
\begin{equation}
\label{E:2.13a}
\II_{\psi}<\infty\quad\Rightarrow\quad 
P\Bigl(\liminf_{t\to\infty}\bigl[Y_t-2\gamma(t)\bigr]\le0\Bigr)=0
\end{equation}
because the \emph{limes inferior} equals positive infinity a.s.
 Denoting $\wh\gamma(r):=\inf_{t\ge \log(1/r)}\gamma(t)$, from  this we obtain
\begin{equation}
	\int_{\log(1/r)}^\infty \rme^{-\gamma(t)/2}\textd t
	\ge \phi(r) \rme^{\wh\gamma(r)/2}
\end{equation}
and
\begin{equation}\label{E:comppsi}
	\int_{\log(1/r)}^\infty \rme^{- 2\gamma(t)}\textd t\le \phi(r) \rme^{-\wh\gamma(r)}
\end{equation}
 once~$r$ is small enough. 
This enables the conclusions of Proposition~\ref{prop-Z-sup}  that  with the help of \twoeqref{E:2.12a}{E:comppsi}  and \eqref{E:2.11a}  yield 
\begin{equation}
\begin{aligned}
\label{E:2.15a}
\II_\psi=\infty\quad\Leftrightarrow\quad
\II_{\frac18\psi}=\infty\quad&\Rightarrow\quad 
P\Bigl(\limsup_{r\downarrow0}\frac{Z^D(B(\wh X,r))}{\int_{\log(1/r)}^\infty \rme^{-\gamma(t)/2}\textd t}\le 1\Bigr) = 0\\
&\Rightarrow\quad P\Bigl(\limsup_{r\downarrow0}\frac{Z^D(B(\wh X,r))}{\phi(r)}\,\rme^{-\wh\gamma(r)/2}\le 1\Bigr) = 0
\end{aligned}
\end{equation}
and
\begin{equation}
\begin{aligned}
\label{E:2.16a}
\II_\psi<\infty\quad\Leftrightarrow\quad
\II_{8\psi}<\infty\quad&\Rightarrow\quad 
P\Bigl(\limsup_{r\downarrow0}\frac{Z^D(B(\wh X,r))}{\int_{\log(1/r)}^\infty \rme^{- 2\gamma(t)}\textd t}\ge 1\Bigr) = 0\\
&\Rightarrow\quad P\Bigl(\limsup_{r\downarrow0}\frac{Z^D(B(\wh X,r))}{\phi(r)}\,\rme^{\wh\gamma(r)}\ge 1\Bigr) = 0.
\end{aligned}
\end{equation}
 These now imply the stated claim. 
\end{proofsect}

The conclusion \eqref{E:2.1a} is designed to feed into a theorem that expands on classical Frostman's criteria for positivity and finiteness of the Hausdorff measure:

\begin{theorem}[Rogers and Taylor]
\label{thm-RT}
For each~$d\ge1$ there is~$c(d)\in(0,\infty)$ such that the following holds for all gauge functions~$\phi$, all Radon measures~$\mu$ on~$\R^d$, all Borel~$A\subseteq\R^d$ and all~$t\in(0,\infty)$:
\begin{equation}
\label{E:2.2a}
\forall x\in A\colon\quad\limsup_{r\downarrow0}\frac{\mu(B(x,r))}{\phi(r)}<t
\end{equation}
implies $\HH^\phi(A)\ge t^{-1}\mu(A)$ and
\begin{equation}
\label{E:2.3a}
\forall x\in A\colon\quad\limsup_{r\downarrow0}\frac{\mu(B(x,r))}{\phi(r)}>t
\end{equation}
implies $\HH^\phi(A)\le c(d)t^{-1}\mu(A)$.
\end{theorem}

\begin{proofsect}{Proof}
The above formulation and a self-contained proof appear as Proposition~6.44 in M\"orters and Peres~\cite{Moerters-Peres} (see also Remark~6.45 thereafter). The statement goes back to Lemmas~2 and~3 in  Rogers and Taylor~\cite{Rogers-Taylor}.
\end{proofsect}

With this in hand, we can give:

\begin{proofsect}{Proof of Theorem~\ref{thm-B} from Proposition~\ref{prop-Z-sup}}
Fix a gauge function  $\psi$  as in the statement and,  with~$\phi$ related to~$\psi$ as in \eqref{E:2.15},  let~$\Omega_\phi$ be the set of measures~$\mu\in\MM_\star(D)$ (where~$\MM_\star(D)$ is as in Lemma~\ref{lemma-2.6}) for which 
\begin{equation}
\label{E:2.7a}
\limsup_{r\downarrow0}\frac{\mu(B(\cdot,r))}{\phi(r)}=\begin{cases}
\infty,\qquad&\text{if }\II_\psi=\infty,
\\
0,\qquad&\text{if }\II_\psi<\infty,
\end{cases}
\end{equation}
holds~$\mu$-almost everywhere.  Since~$\phi$ is continuous, Lemma~\ref{lemma-2.6} implies that $\Omega_\phi$  is a Borel subset of~$\MM(\R^2)$ and Lemma~\ref{lemma-Z-sup} then gives $P(Z^D\in\Omega_\phi)=1$. 

Let us assume~$Z^D\in\Omega_\phi$ for the rest of the proof.
For $a=0,\infty$ denote 
\begin{equation}
\label{E:2.21}
D_a:=\biggl\{x\in D\colon\limsup_{r\downarrow0}\frac{Z^D(B(x,r))}{\phi(r)}=a\biggr\}
\end{equation}
and let $\wt D:=D\smallsetminus(D_0\cup D_\infty)$. For $t\in(0,\infty)$ consider also
\begin{equation}
\wt D_t:=\biggl\{x\in D\colon\limsup_{r\downarrow0}\frac{Z^D(B(x,r))}{\phi(r)}\in(t,\infty)\biggr\}.
\end{equation}
Lemma~\ref{lemma-2.6} ensures that these are Borel sets and  by our assumption that $Z^D\in\Omega_\phi$ we have   $Z^D(\wt D_t)=0$ for every~$t\in(0,\infty)$. From \eqref{E:2.3a} we thus get
\begin{equation}
\HH^\phi(\wt D_t)\le c(2)t^{-1}Z^D(\wt D_t)=0,\quad t\in(0,\infty).
\end{equation}
As~$\wt D=\bigcup_{t\in\Q\cap(0,\infty)}\wt D_t$, we  conclude  $\HH^\phi(\wt D)=0$. 

Next consider the set~$D_\infty$. Here \eqref{E:2.3a} shows
\begin{equation}
\HH^\phi(D_\infty)\le c(2)\,t^{-1}Z^D(D_\infty),\quad t<\infty.
\end{equation}
Since $Z^D$ is finite, taking $t\to\infty$ we infer $\HH^\phi(D_\infty)=0$ and so $\HH^\phi(D_0^\cc)=0$.
For any $A\subseteq D$ Borel, \eqref{E:2.2a} then gives
\begin{equation}
\HH^\phi(A)=\HH^\phi(A\cap D_0)\ge t^{-1}Z^D(A\cap D_0),\quad t>0.
\end{equation}
This implies
\begin{equation}
\label{E:2.9}
Z^D(A\cap D_0)>0\,\,\Rightarrow\,\,\HH^\phi(A)=\infty
\end{equation}
by taking~$t\downarrow0$.

We now conclude the claim: If $\II_\psi<\infty$, then $Z^D(D_0^\cc)=0$ by the second line in \eqref{E:2.7a} and so by \eqref{E:2.9}, $Z^D(A)>0$ forces $\HH^\phi(A)=\infty$ for any~$A\subseteq D$ Borel, proving~(1). For~(2) note that, under $\II_\psi=\infty$, the first line in \eqref{E:2.7a} gives $Z^D(D_\infty^\cc)=0$. The set $A:=D_\infty$ is thus a $Z^D$-carrier for which, as shown above, $\HH^\phi(A)=0$. 
\end{proofsect}

We also give:

\begin{proofsect}{Proof of Corollary~\ref{cor-1.2} from Proposition~\ref{prop-Z-sup}}
Continuing the argument from the previous proof, for  $\psi(t):=[\log(1+t)]^{-\theta}$ with~$\theta>0$ we have $\phi$ as in \eqref{E:1.6u} and~$\II_\psi=\infty$ as soon as~$\theta<1$. Set~$\theta:=1/2$ and let $D_\infty$ be defined by the corresponding gauge function~$\phi$.  The fact that $\phi$ decays to zero slower than any power then yields $\HH^{\phi_s}(D_\infty)\le \HH^{\phi}(D_\infty)$ for all~$s>0$. But $\psi$ is decreasing and satisfies $\II_\psi=\infty$ so, by the above arguments,~$D_\infty$ is a $Z^D$-carrier with $\HH^{\phi}(D_\infty)=0$.  In particular,  $\dim_{\rm H}(D_\infty)=0$.
\end{proofsect}

Proposition~\ref{prop-Z-sup} encapsulates the ``continuum'' part of our proof; in the remainder of the paper we resort to discretizations based on the Discrete Gaussian Free Field. An alternative strategy could rely on continuum regularizations of the CGFF (e.g., via circle averages as in~\cite{DRSV1,DRSV2} or in~\cite{JS}) throughout.

\section{Random walk representation}
\label{sec-3}\noindent
 Here, on our way to the  proof  of Proposition~\ref{prop-Z-sup}, we present a reduction of the quantities in the statement to objects involving the Discrete Gaussian Free Field (DGFF) in lattice approximations of~$D$; see Lemma~\ref{lemma-3.1} and \ref{lemma-3.3}. We then recall the concentric decomposition of the DGFF which leads to a random walk representation of the relevant quantities; see Lemmas~\ref{l:concdec} and~\ref{lemma-rw}. 

\subsection{Reduction to DGFF quantities}
\label{sec-3.1}\noindent
Assume~$D\in\mathfrak D$ to be fixed for the rest of the paper. Given any integer~$N\ge1$ and writing $\dist$ for the $\ell^\infty$-distance on~$\Z^2$ or on $\R^2$, we will take
\begin{equation}
\label{e:discretization}
D_N := \bigl\{x \in \Z^2 \colon \dist(x/N,D^\cc)>\ffrac{1}{N}\bigr\}
\end{equation}
for a lattice approximation of~$D$ at scale~$N$.  This is a canonical choice that ensures a close correspondence of harmonic functions in~$D$ and~$D_N$; see \cite[Appendix]{BL2}.  It also enables  the conclusions of~\cite{BGL} that we rely on in  our  proofs. 

Given~$V\subsetneq\Z^d$, let $G^{V}\colon V\times V\to[0,\infty)$ denote the Green function on~$V$ with~$G^{V}(x,y)$ defined as the expected number of visits to~$y$ by the simple random walk on~$\Z^2$ started from~$x$ and killed upon the first exit from~$V$. Let~$h^{V}$ denote a sample of the DGFF on~$V$ which is a centered Gaussian process with covariance $G^{V}(\cdot,\cdot)$. We will henceforth write~$\BbbP$ and~$\E$ for the law of the DGFF and its expectation while reserving ``plain''~$P$ and~$E$ for probability and expectation in other contexts.  Of prime interest is the choice $V:=D_N$ but other choices will be needed as well. 

For the lattice counterpart of~$Z^D$ we then set
\begin{equation}
\label{E:3.2}
Z^D_N:=\frac1{\log N}\sum_{x\in D_N}\texte^{\alpha(h^{D_N}_x-m_N)}\delta_{x/N},
\end{equation}
where  $\alpha:=2/\sqrt{g}$ for $g:=2/\pi$ and 
\begin{equation}
\label{E:3.3}
m_N:=2\sqrt g\log N-\frac34\sqrt g\log\log(N\vee\texte).
\end{equation}
Notice that~$Z^D_N$ is a Radon measure concentrated on~$D$. As shown in Rhodes and Vargas~\cite[Theorem~5.13]{RV-review}, Junnila and Saksman~\cite[Theorem~1.1]{JS} and again in the recent study~\cite{BGL} by the present authors, 
\begin{equation}
\label{E:3.4}
Z^D_N\,\,\underset{N\to\infty}\Lawarrow\,\, Z^D,
\end{equation}
where the convergence in law is relative to the vague topology on~$\MM(\R^2)$.

\begin{remark}
\label{rem-3.1}
 Let us highlight some facts about weak convergence in~$\MM(\R^2)$.
Abbreviate $\langle\mu, f\rangle:=\int\mu(\textd x)f(x)$. The key point to note is that,  given random elements $\mu$ and $\{\mu_N\}_{N\ge1}$ from~$\MM(\R^2)$, the weak convergence~$\mu_N\Lawarrow\mu$ is implied by (and is thus equivalent to) the convergence of real-valued random variables $\langle \mu_N,f\rangle\Lawarrow\langle\mu,f\rangle$ for all continuous $f\colon\R^2\to\R_+$ with compact support. Indeed, the latter implies convergence of the Laplace transforms of these random variables for all~$f$ as stated. Linearity along with the Cram\'er-Wold device then extend this to the convergence
\begin{equation}
\bigl(\langle\mu_N,f_1\rangle,\dots,\langle \mu_N,f_k\rangle\bigr)
\,\,\underset{N\to\infty}\Lawarrow\,\,\,\bigl(\langle\mu,f_1\rangle,\dots,\langle \mu,f_k\rangle\bigr)
\end{equation}
as random vectors in~$\R^k$ for all $k$-tuples $f_1,\dots,f_k$ of continuous functions with compact support. 
The space~$\MM(\R^2)$ is second countable and there exists a sequence~$\{f_i\}_{i\ge1}$ of continuous functions with compact support such that each open $U\subseteq\MM(\R^2)$ is an increasing limit of sets of the form $U_k:=\bigcap_{i=1}^k\{\nu\in\MM(\R^2)\colon \langle \nu,f_i\rangle\in O_i\}$ for some $k$-dependent $k$-tuple of open sets $O_1,\dots,O_k\subseteq\R$.  It follows that  
\begin{equation}
\begin{aligned}
\liminf_{N\to\infty} P\bigl(&\mu_N\in U\bigr)\ge \liminf_{N\to\infty} P\bigl(\mu_N\in U_k\bigr)\\
&=\liminf_{N\to\infty}P\biggl( \bigl(\langle\mu_N,f_1\rangle,\dots,\langle \mu_N,f_k\rangle\bigr)\in\bigtimes_{i=1}^k O_i\biggr)
\\
&\ge P\biggl( \bigl(\langle\mu,f_1\rangle,\dots,\langle \mu,f_k\rangle\bigr)\in\bigtimes_{i=1}^k O_i\biggr)
=P(\mu\in U_k)\,\underset{k\to\infty}\longrightarrow\,P(\mu\in U),
\end{aligned}
\end{equation}
where the middle inequality follows from one of the equivalent definitions (per Portmanteau's Theorem) of weak convergence in~$\R^k$. The Portmanteau Theorem again (which applies since~$\MM(\R^2)$ is completely metrizable)  then gives  $\mu_N\Lawarrow\mu$.
\end{remark}

 We will henceforth use
\begin{equation}
B^\infty(x,r):=[x-r,x+r]^2
\end{equation}
 to denote the $\ell^\infty$-ball of radius~$r$ centered at~$x$. Let us also write 
\begin{equation}
\wh Z^D_N(\cdot):=\frac{Z^D_N(\cdot)}{Z^D_N(D)}
\end{equation}
for the normalized counterpart of~$Z^D_N$  and  let~$\wh X_N$ be a sample from~$\wh Z^D_N$ conditional on~$h^{D_N}$. This means that the joint distribution of $(\wh X_N,Z_N^D)$ on $\bbR^2\times\mathcal M(\bbR^2)$ obeys
	\begin{equation}
		\label{e:ZXN}
		P(\wh X_N\in A, Z^D_N\in B)=E\bigl(\wh Z^D_N(A)1_{\{Z^D_N\in B\}}\bigr)
	\end{equation}
for all Borel~$A\subseteq\R^2$ and~$B\subseteq\MM(\R^2)$.
The convergence \eqref{E:3.4} can then be augmented to: 

\begin{lemma}
	\label{lemma-3.1}
 Relative to the product topology on~$\R^2\times\MM(\R^2)$, 
	\begin{equation}
		\label{E:3.6a}
		(\wh X_N,Z^D_N)\,\,\underset{N\to\infty}\Lawarrow\,\, (\wh X, Z^D).
	\end{equation}
Moreover,  there exists a dense subset~$\Sigma\subseteq(0,\infty)$ such that for all $k\ge1$ and all $r_1,\dots,r_k\in\Sigma$, 
	\begin{equation}
		\label{E:3.8a}
		\begin{aligned}
		\Bigl(Z^D_N\bigl( B^\infty(\wh X_N,r_1)\bigr),&\dots,Z^D_N\bigl( B^\infty(\wh X_N,r_k)\bigr)\Bigr)
		\\&\underset{N\to\infty}\Lawarrow\,\, 
		\Bigl(Z^D\bigl( B^\infty(\wh X,r_1)\bigr),\dots,Z^D\bigl( B^\infty(\wh X,r_k)\bigr)\Bigr)
		\end{aligned}
	\end{equation}
in the sense of convergence in law on~$\R^k$. 
\end{lemma}

\begin{proofsect}{Proof}  Continuing the use of the shorthand $\langle\mu, f\rangle:=\int\mu(\textd x)f(x)$, the argument in Remark~\ref{rem-3.1} shows that for \eqref{E:3.6a} it suffices to prove
\begin{equation}
\bigl(\wh X_N,\langle Z^D_N,f\rangle\bigr)\,\,\underset{N\to\infty}\Lawarrow\,\,\bigl(\wh X,\langle Z^D,f\rangle\bigr)
\end{equation}
for each~$f\colon\R^2\to\R_+$ with compact support. This will in turn follow if
\begin{equation}
\E\Bigl(g_1(\wh X_N)g_2\bigl(\langle Z^D_N,f\rangle\bigr)\Bigr)\,\,\underset{N\to\infty}\longrightarrow\,\,\E\Bigl(g_1(\wh X)g_2\bigl(\langle Z^D,f\rangle\bigr)\Bigr)
\end{equation}
for all bounded and continuous~$g_1\colon\R^2\to\R_+$ and ~$g_2\colon\R\to\R_+$. 

Let~$g_3\colon\R^2\to[0,1]$ be continuous with compact support  such that~$g_3=1$ on~$D$. Define $\varphi\colon\R^3\to\R_+$ by
	\begin{equation}
		\varphi(x_1,x_2,x_3):=\max\biggl\{\frac{|x_1|}{|x_2|},1\biggr\}g_2(x_3)
	\end{equation}
when~$x_2\ne0$ and by~$\varphi(x_1,x_2,x_3):=1$ when~$x_2=0$. Assuming without loss of generality that $|g_1|\le1$, we then have 
\begin{equation}
\begin{aligned}
E\Bigl(&g_1(\wh X_N)g_2\bigl(\langle Z^D_N,f\rangle\bigr)\Bigr)
		=\bbE\Bigl(\langle \wh Z^D_N,g_1\rangle g_2\bigl(\langle Z^D_N,f\rangle\bigr)\Bigr)
		\\
		&=\bbE \biggl(\varphi\Bigl( \langle Z^D_N,g_1\rangle,\langle Z^D_N,g_3\rangle, \langle Z^D_N,f\rangle \Bigr)\biggr)
		\,\underset{N\to\infty}\longrightarrow\,
		E \biggl(\varphi\Bigl( \langle Z^D,g_1\rangle,\langle Z^D,g_3\rangle, \langle Z^D,f\rangle \Bigr)\biggr)
		\\
		&=	
		E\Bigl(\langle \wh Z^D,g_1\rangle g_2\bigl(\langle Z^D,f\rangle\bigr)\Bigr)	=
		E\Bigl(g_1(\wh X)g_2\bigl(\langle Z^D,f\rangle\bigr)\Bigr).
\end{aligned}
\end{equation}
Here the limit follows from \eqref{E:3.4} and the fact that $\mu\mapsto \varphi(\langle \mu,g_1\rangle,\langle \mu,g_3\rangle, \langle \mu,f\rangle)$ is boun\-ded and continuous except on the set where $\langle \mu,g_3\rangle=0$ which is a closed set of zero measure under the law of~$Z^D$. Hence~\eqref{E:3.6a} follows. 

Moving to the proof of \eqref{E:3.8a}, write~$A_\delta(x,r):= B^\infty(x,r+\delta)\smallsetminus B^\infty(x,r-\delta)$ and approximate  the  indicator~$1_{B^\infty(x,r)}$ by compactly-sup\-por\-ted continuous functions $f,g\colon\R^2\to[0,1]$ such that~$f\le 1_{B^\infty(x,r)}\le g$ and~$g-f\le 1_{A_\delta(x,r+\delta)}$. Invoking \eqref{E:3.6a} for~$N\to\infty$ and noting that
\begin{equation}
\lim_{\delta\downarrow0}P\Bigl(Z^D\bigl(A_\delta(\wh X,r)\bigr)>0\Bigr)=P\Bigl(Z^D(\partial B^\infty(\wh X,r)\bigr)>0\Bigr)
\end{equation}
we get convergence \eqref{E:3.8a} whenever~$r_1,\dots,r_k$ belong to the set 
\begin{equation}
\Sigma:=\biggl\{r>0\colon P\Bigl(Z^D(\partial B^\infty(\wh X,r)\bigr)>0\Bigr)=0\biggr\}.
\end{equation}
This set is dense in~$(0,\infty)$ because it contains every continuity point of the monotone function $r\mapsto E(Z^D(B^\infty(\wh X,r))\texte^{-Z^D(D)})$.
\end{proofsect}

We note that we actually believe that $Z^D(\partial B^\infty(\wh X,r))=0$ holds for all~$r>0$ in a.e.\ sample of~$Z^D$. However, the above is sufficient for our needs and proving a stronger statement would require computations that we prefer to avoid.

\subsection{ Asymptotic statements and bounds}
Lemma~\ref{lemma-3.1} will be used to recast probabilities involving~$Z^D$ and~$\wh X$ as limits of probabilities involving their ``discretized'' counterparts~$Z^D_N$ and~$\wh X_N$. Manipulations with these objects will be facilitated by the following facts:

\begin{lemma}
	\label{lemma-3.2}
	Let~$h^{D_N}$ and~$\wh X_N$ be as above. For all~$a\ge0$ we have
	\begin{equation}
		\label{E:3.8}
		\lim_{N\to\infty}P\Bigl(h^{D_N}(N\wh X_N)-m_N\not\in\bigl[-a\sqrt{\log N},0\bigr]\Bigr)=\texte^{-\frac{a^2}{2g}}.
	\end{equation}
	In addition, 
	\begin{equation}
		\label{E:3.10a}
		\lim_{\delta\downarrow0}\limsup_{N\to\infty}P\bigl(\dist(\wh X_N,D^\cc)>\delta\bigr)=0.
	\end{equation}
\end{lemma}

\begin{proofsect}{Proof}
	Consider the near extremal process associated with DGFF in~$D_N$
	\begin{equation}
		\zeta_N^D:=\frac1{\log N}\sum_{x\in D_N}\texte^{\alpha(h^{D_N}_x-m_N)}\delta_{x/N}\otimes\delta_{(m_N-h^{D_N}_x)/\sqrt{\log N}}.
	\end{equation}
	By the main result of Biskup, Gufler and Louidor~\cite{BGL},
	\begin{equation}
		\zeta^N_D\,\,\underset{N\to\infty}\Lawarrow\,\,c_\star\,Z^D\otimes 1_{[0,\infty)}(t)t\texte^{-\frac{t^2}{2g}}\textd t
	\end{equation}
	where $c_\star\in(0,\infty)$  is a constant.
	The convergence is relative to the vague topology on $\overline D\times\overline\R$, where $\overline\R$ denotes the two-point compactification of $\R$.
	Introducing the normalized process $\wh\zeta^D_N(\cdot):=\zeta^D_N(\cdot)/\zeta^D_N(D\times\R)$, we get
	\begin{equation}
		\wh\zeta^D_N\bigl(D\times [0,a]^\cc\bigr)\,\,\,\underset{N\to\infty}\Lawarrow\,\,\,\frac{\displaystyle\int_a^\infty t\texte^{-\frac{t^2}{2g}}\textd t}{\displaystyle\int_0^\infty t\texte^{-\frac{t^2}{2g}}\textd t}.
	\end{equation}
	Since the probability in the statement is exactly the expectation of the left-hand side, the claim follows by a direct calculation.
	
	 As to  \eqref{E:3.10a},  let $D^\delta:=\{x\in\R^2:\dist(x,D^\cc)>\delta\}$ and note that  the probability is bounded by $\E\wh Z^D_N(\R^2\smallsetminus D^\delta)$ whose \emph{limes superior} as~$N\to\infty$ is bounded by $E\wh Z^D(\R^2\smallsetminus D^\delta)$. This tends to zero as~$\delta\downarrow0$ by the fact that~$Z^D$ is supported on~$D$ a.s.
\end{proofsect}

The right-hand side of \eqref{E:3.8} can be made as small as desired by taking~$a$ large. In computations of probabilities involving~$Z^D_N$ and~$\wh X_N$, we can thus first restrict $h^{D_N}(N\wh X_N)$ to the interval $m_N+[-a\sqrt{\log N},0]$, force~$\wh X_N\in D_N^\delta$ for~$\delta>0$ small, where
\begin{equation}
	D_N^\delta:=\bigl\{x\in D_N\colon\dist(x,D_N^\cc)>\delta N\bigr\},
\end{equation}
and then work under conditioning on~$\wh X_N$ and~$h^{D_N}(N\wh X_N)$. We will summarize the elementary calculations underpinning this step in the next lemma but before we do that, recall the asymptotic formula for~$G^{D_N}(x)$ from~\cite[Theorem~1.17]{B-notes},
\begin{equation}
	\label{E:3.15a}
	G^{D_N}(x,x) = g\log N + g\log r^D(x/N)+c_0+o(1),
\end{equation}
where~$o(1)\to0$ as~$N\to\infty$ uniformly in~$x\in D_N^\delta$. Here~$c_0$ is a (known) constant and $r^D\colon D\to(0,\infty)$ is (for~$D$ simply connected) the conformal radius of~$D$ from~$x$.

\begin{lemma}
	\label{lemma-3.3}
	Let~$\delta>0$. For all~$x\in D_N^\delta$, all~$a>0$ and all Borel sets~$B\subseteq\R^{D_N}$, 
	\begin{equation}
		\begin{aligned}
			\label{E:3.9}
			P\Bigl(\wh X_N=&x/N,\,0\le m_N-h_x^{D_N}\le a\sqrt{\log N},\,h^{D_N}\in B\Bigr)
			\\
			&=\texte^{o(1)+\frac{c_0}{2g}}\frac{\sqrt{\log N}}{N^2}\,r^D(x/N)^2
			\\
			&\qquad\times \frac{1}{\sqrt{2\pi g}} \int_{[0,a]}\textd t \,\texte^{-\frac{t^2}{2g}}\,\E\Bigl(Z^D_N(D)^{-1}1_{\{h^{D_N}\in B\}}\Big|\,m_N-h_x^{D_N}=t\sqrt{\log N}\Bigr),
		\end{aligned}
	\end{equation}
	where~$o(1)\to0$ as~$N\to\infty$ uniformly in~$B\in\R^{D_N}$, $x\in D_N^\delta$ and~$a\in[0,(\log N)^{1/2-\delta}]$.
\end{lemma}

\begin{proofsect}{Proof}
	Let~$f_{N,x}$ denote the probability density of~$h^{D_N}_x$ with respect to the Lebesgue measure.  The left-hand side in~\eqref{E:3.9} can be written as
	\begin{equation}
	\label{E:3.26i}
	\begin{aligned}
		\sqrt{\log N}\int_{[0,a]}\textd t\, f_{N,x}&\bigl(m_N - t\sqrt{\log N}\bigr)
		\\
		&\times\E\biggl(\frac{Z^D_N(x/N)}{Z^D_N(D)}\,1_{\{h^{D_N}\in B\}}\bigg|\,m_N-h_x^{D_N}=t\sqrt{\log N}\biggr)
		\end{aligned}
	\end{equation}
	by~\eqref{e:ZXN} and a change of variables. Using \eqref{E:3.15a} along with the explicit form of~$m_N$ from \eqref{E:3.3} we get, for $o(1)\to0$ uniformly in~$t\in[0,a]$ and~$x\in D_N^\delta$,
	\begin{equation}
		\label{e:Gauss-3.3}
		f_{N,x}\bigl(m_N-t\sqrt{\log N}\bigr) = \texte^{o(1)+\frac{c_0}{2g}}\frac{1}{\sqrt{2\pi g}}\,\frac{\log N}{N^2}\,r^D(x/N)^2\,\texte^{-\frac{t^2}{2g}+\alpha t\sqrt{\log N}},
	\end{equation}
	as shown in~\cite[Lemma~5.5]{BGL}. 
	Combining with the term $(\log N)^{-1}\texte^{-\alpha t\sqrt{\log N}}$ arising from \eqref{E:3.2}, we get \eqref{E:3.9}.
\end{proofsect}

With Lemma~\ref{lemma-3.3} in hand,  the  computations are effectively reduced to  asymptotic of  the conditional expectation in \eqref{E:3.9}. Next we recall:

\begin{lemma}
	\label{lemma-DGFF-tail}
	There exists a constant $c<\infty$ such that for all $N\ge1$ and all $u\in[1,\sqrt{\log N})$,
	\begin{equation}
		\label{E:3.18}
		\BbbP\bigl(\max h^{D_N}> m_N+u\bigr)\le c u\texte^{-\alpha u}
	\end{equation}
\end{lemma}

\begin{proof}
	For~$D_N:=(0,N)^2\cap\Z^2$, this is~\cite[Theorem~1.4]{DZ}. A routine extension to general domains is performed in \cite[Proposition~3.3]{BL2}.
\end{proof}

The event~$h^{D_N}\le m_N+u$ is thus typical and we can make it part of the event $h^{D_N}\in B$ in \eqref{E:3.9} at the cost of additive corrections that tend to zero as~$u\to\infty$. This is useful in light of  the following asymptotic: 

\begin{lemma}
	\label{lemma-3.4}
	Let $u>0$ and let $\delta>0$ be sufficiently small. Then~$O(1)$ defined by
	\begin{equation}
		\label{e:3.19}
		\BbbP\Bigl(h^{D_N}\le m_N+u\,\Big|\, m_N-h_x^{D_N}=t\sqrt{\log N}\Bigr)=\frac{\texte^{O(1)}}{\sqrt{\log N}}\,(1+t)
	\end{equation}
	is bounded uniformly in~$N\ge3$, $x\in D_N^\delta$ and~$t\in[0,(\log N)^{1/2-\delta}]$.
\end{lemma}

\begin{proof}
	  Abbreviating $U:=D-x/N$ and using $U_N$ to denote its discretization via~\eqref{e:discretization}, the  left-hand side of~\eqref{e:3.19} equals  (using the notation of \cite{Ballot}) 
	\begin{equation}
		\BbbP\Bigl(h^{ U_N} \le 0\,\Big|\, h_0^{ U_N}=-t\sqrt{\log N}-u,\,h^{U_N}|_{\partial U_N}=-m_N-u\Bigr).
	\end{equation}
	The assertion now follows from~\cite[Corollary~1.7 and Theorem~1.8]{Ballot}, with the outer domain  $U_N$  and the inner domain $V:=B(0,1/4)$ discretized by  of~\cite[Eq.~(1.3)]{Ballot}  so that $V_0^-=\{0\}^\cc$ (using also~\cite[Remark~2]{Ballot}).
\end{proof}

 Thanks to Lemma~\ref{lemma-3.4}, whenever $h^{D_N}\le m_N+u$ is included in the event $h^{D_N}\in B$ in \eqref{E:3.26i}, one can add it to the conditional event and effectively absorb the prefactor~$\sqrt{\log N}$ into the conditional expectation. 
Introducing
\begin{equation}
\label{E:3.20a}
\Gamma_{N,u}(x,t):=\bigl\{h^{D_N}\le m_N+u\}\cap\bigl\{m_N-h^{D_N}_x = t\sqrt{\log N}\bigr\},
\end{equation}
this  will reduce  the proof of Proposition~\ref{prop-Z-sup} to upper bounds on the expectation
\begin{equation}
\label{E:3.15}
\E\Bigl(Z^D_N(D)^{-1}1_{\{h^{D_N}\in B\}}\Big|\,\Gamma_{N,u}(x,t)\Bigr)
\end{equation}
for the relevant choices of the event $h^{D_N}\in B$  (see \eqref{E:4.6} or \eqref{E:4.31i})  uniformly in~$x\in D_N^\delta$, $t\in[0,a]$ and $u\ge1$. While the conditioning is singular, the expression is determined  by the fact that~$h^{D_N}_x$ has a continuous probability density.

\subsection{Concentric decomposition}
\label{sec:conc}\noindent
Our way of controlling the expectation \eqref{E:3.15} relies on a representation of the field~$h^{D_N}$ as the sum of independent contributions that are associated with concentric annuli centered at some fixed point of~$D_N$. 

We start by defining the relevant geometric notions. Fix~$\delta\in(0,1)$ small, $N\ge1$ large and let~$x_0\in D_N^\delta$. We will keep~$x_0$ fixed throughout the developments. Although the objects introduced below depend on~$x_0$ and~$\delta$, we will not make this dependence notationally explicit. Next let us abbreviate
\begin{equation}
\label{E:3.22a}
\kappa(\delta):=\inf\bigl\{k\in\N\colon \texte^{-k}<\delta\bigr\}
\end{equation}
and let
\begin{equation}
\label{E:3.16}
n:=\inf\bigl\{k\in\N\colon N\texte^{-\kappa(\delta)-k}<1\bigr\}.
\end{equation}
Note that $|n-\log N|$ is bounded by a constant that depends only on~$\delta$ and~$D$.
For $k=1,\dots,n-1$, let
\begin{equation}
\label{E:3.17}
\Delta^k:=\bigl\{x\in \Z^2\colon\dist(x_0,x)<N\texte^{-\kappa(\delta)-k}\bigr\}
\end{equation}
and note that
\begin{equation}
 D_N=:\Delta^0\supseteq\Delta^1\supseteq\ldots
\supseteq\Delta^{n-1}
\supseteq\Delta^{n}:=\{x_0\}\supseteq \Delta^{n+1}:=\emptyset.
\end{equation}
We then have:

\begin{lemma}[Concentric decomposition]
\label{l:concdec}
There exist functions $\{\frakb_k\colon k=0,\dots,n\}$ on~$\Z^2$, centered Gaussians $\{\varphi_k(x_0)\colon k=0,\dots,n\}$ and centered Gaussian fields $\{\chi_k\colon k=0,\dots,n\}$ and $\{h_k'\colon k=0,\dots,n\}$ on~$\Z^2$ such that all the objects in
\begin{equation}
\label{E:3.19}
\bigl\{\varphi_k(x_0)\colon k=0,\dots,n\bigr\}\cup\{\chi_k\colon k=0,\dots,n\}\cup\{h_k'\colon k=0,\dots,n\}
\end{equation}
are independent with $h_{n}':=0$ and, for $k=0,\ldots, n-1$, 
\begin{equation}
\label{E:3.20}
h_k' \,\,\laweq\,\,\,\text{\rm DGFF on }\Delta^{k}\smallsetminus\overline{\Delta^{k+1}}
\end{equation}
and such that, for all $k=1,\dots,n$,
\begin{equation}
\label{E:3.21}
h^{\Delta_k}
\,\,\laweq\,\, 
\sum_{j=k}^{n}\bigl[(1+\fb_j(\cdot))\varphi_j(x_0)+\chi_j+h'_j\bigr].
\end{equation}
In addition, $\chi_k(x_0)=0$ a.s.\ and $\fb_k(x_0)=0$ for all~$k=0,\dots,n$. With these added, the functions $\{\fb_k\colon k=0,\dots,n\}$ and the laws of all the objects in \eqref{E:3.19} are determined uniquely,  and we have $\min \fb_k \ge -1$ for all $k=0,\dots,n$. Moreover, $\chi_k$ and $1+\fb_k$ vanish a.s.\ in~$\Z^2\smallsetminus\Delta^k$.
\end{lemma}

\begin{proofsect}{Proof}
This paraphrases~\cite[Proposition~3.3]{BL3} while leaving out the explicit form of the law of the~$\chi_k$'s and the formulas for the~$\fb_k$'s given there. To see that these objects are uniquely determined by the stated facts, note that
\begin{equation}
\label{E:3.22}
\varphi_k(x):=(1+\fb_k(x))\varphi_k(x_0)+\chi_k(x)
\end{equation}
is a centered Gaussian field independent of the~$h_k'$'s with law determined from \eqref{E:3.19} and \eqref{E:3.21} by a covariance calculation. (The potential clash of notations in \eqref{E:3.22} for $x=x_0$ is eliminated by the assumption~$\chi_k(x_0)=0$ and $\fb_k(x_0)=0$.) Another covariance calculation based on \eqref{E:3.22} and \eqref{E:3.19} then determines~$\fb_k$ and the law of~$\chi_k$. 
\end{proofsect}

Note that \eqref{E:3.21} provides a natural coupling of all $\{h^{\Delta_k}\colon k=0,\dots,n\}$. This permits us to treat \eqref{E:3.21} as an almost-sure equality in the sequel. A key point is that, in the situations of interest, the above decomposition can largely be controlled by the following quantities
\begin{equation}
\label{E:3.29}
S_k:=\sum_{j=0}^{k}\varphi_j(x_0),\quad k=0,\ldots,n,
\end{equation}
which in light of the independence of $\{\varphi_j(x_0)\colon j=0,\dots,n\}$ will be referred to as a \emph{random walk}. (Note, however, the inhomogeneity of the step distribution.) Moreover, 
\begin{equation}
\label{E:3.24}
h^{D_N}(x_0)=S_n
\end{equation}
by \eqref{E:3.21} and the fact that $\fb_k(x_0)=0=\chi_k(x_0)=h_k'(x_0)$ for all~$k$ as above, so conditioning on~$h^{D_N}(x_0)$ is the same as conditioning on~$S_n$. The values of $h^{D_N}$ at other points of~$D_N$ can be expressed by adding suitable ``decorations'' to the random walk:

\begin{lemma}[Random walk representation]
\label{lemma-rw}
Regarding \eqref{E:3.21} as an a.s.\ equality,
\begin{equation}
\label{E:3.27}
h^{D_N}_x=S_k+\sum_{j=0}^{k}
[\fb_j(x)\varphi_j(x_0)+\chi_j(x)+h'_j(x)]
\end{equation}
holds for all $x\in\Delta^{k}\smallsetminus\Delta^{k+1}$ and all $k=0,\ldots,n$.
\end{lemma}

\begin{proofsect}{Proof}
Since both~$h^{\Delta_k}$ and, by \eqref{E:3.20}, also~$h_k'$ vanish outside~$\Delta^k$, a variance calculation shows that the same applies to~$\chi_k$ and $1+\fb_k$.
The claim follows from \eqref{E:3.21} for~$k:=n$.
\end{proofsect}

\begin{remark}
A covariance calculation along with \eqref{E:3.19} and \eqref{E:3.21} show that, for each $1\le k\le n$, the field \eqref{E:3.22} obeys 
\begin{equation}
\varphi_k \,\laweq\, E\Bigl(h^{\Delta^k}\,\Big|\,\sigma\bigl(h^{\Delta^k}(x)\colon x\in\partial\Delta^{k+1}\bigr)\Bigr).
\end{equation}
From the a.s.~version of \eqref{E:3.21} we in turn get
\begin{equation}
\label{E:3.27i}
h^{\Delta^k}=h^{\Delta^{k+1}}+h_k'+\varphi_k.
\end{equation}
Since $h^{\Delta^{k+1}}+h_k'$ has the law of the DGFF on~$\Delta^k\smallsetminus\partial\Delta^{k+1}$, \eqref{E:3.27i} is an expression of the Gibbs-Markov property of the DGFF   (see e.\,g.~\cite[Lemma~3.1]{B-notes}). This property underlies many technical arguments involving the DGFF.

As for~$\chi_k$, here \eqref{E:3.22} shows
\begin{equation}
\chi_k = E\bigl(\varphi_k\,\bigl|\,\varphi_k(x_0)=0\bigr).
\end{equation}
Regarding~$\varphi_k(x_0)$ as a constant field representing~$\varphi_k$, the ``decoration'' $\fb_k\varphi_k(x_0)+\chi_k$ accounts for the spatial variations of~$\varphi_k$.

By~\eqref{E:3.29}, \eqref{E:3.27i} and the Gibbs-Markov property, we also have
\begin{equation}\label{e:harm-av}
	S_k = \sum_{i={ 0}}^k\varphi_i(x_0) = \ol{(h^{D_N})_{\partial \Delta^k}}(x_0),
\end{equation}
where $x\mapsto \ol{f_U}(x)$  henceforth denotes the unique bounded  harmonic extension to $\bbZ^2$ of the restriction of a function~$f$ to a finite subset $U\subseteq \bbZ^2$.
\end{remark}

As the reader may have noticed, while we follow closely~\cite{BL3}, our domains \eqref{E:3.17} are labeled in reverse order which causes changes in the labeling of the sums in \eqref{E:3.21} and \eqref{E:3.27} compared to~\cite{BL3}. This is because our main focus here is on the macroscopic aspects of the field, rather than the local behavior that \cite{BL3} was primarily concerned with.

\section{Proof of the key proposition}
\label{sec-4}\noindent
In this section we  employ the  random walk representation of the DGFF to give a proof of Proposition~\ref{prop-Z-sup}  conditional on three lemmas; namely, Lemma~\ref{lemma-3.9},~\ref{lemma-3.10} and~\ref{l:bessel} below. These lemmas will be proved in Section~\ref{sec-5}.

\subsection{Control variable}
\label{sec-3.3}\noindent
In order to make efficient use of the random walk representation, we have to control the decorations; namely, the sums on the right-hand side of \eqref{E:3.27}. 
Here we denote
\begin{equation}
	\label{e:Xi-k}
	\Xi_k:=\log\biggl(\frac1{\log N}\sum_{x\in\Delta^k\smallsetminus\Delta^{k+1}}\exp\Bigl\{\sum_{j=0}^k\alpha[\fb_j(x)\varphi_j(x_0)+\chi_j(x)]+\alpha(h_k'-m_{N\texte^{-k}})\Bigr\}\biggr).
\end{equation}
For $k\ge 1$,  abbreviate  
\begin{equation}
B^\infty_k:=B^\infty( x_0/N,\texte^{-k-\kappa(\delta)})=\{x\in\bbR^2\colon \dist(x, x_0/N )<\texte^{-k-\kappa(\delta)}\}
\end{equation}
 and abbreviate $A_k:=B^\infty_k\smallsetminus B^\infty_{k+1}$.
 Note that then  $\Delta^k = \{x\in\Z^2\colon x/N\in B^\infty_k\}$  and  $\Delta^k \smallsetminus\Delta^{k+1} = \{x\in\Z^2\colon x/N\in A_k\}$.  Using these objects we get 
\begin{equation}
	\label{E:3.39}
	Z^D_N(A_k)=\texte^{\alpha[S_k-m_N+m_{N\texte^{-k}}]+\Xi_k}
\end{equation}
for all $k=1,\ldots,n$. 

As in earlier work~\cite{BL3}, we will control the various ``error'' terms on the right-hand side of \eqref{E:3.39} by a single random variable:
\begin{definition}[Control variable]
	Fix~$\eta\in(0,1/4)$, $\ell\in\N$, and let~$K=K(\eta,\ell)$ be the smallest number in $\{0,1,\dots,\ell\}$ such that
	\begin{equation}
		\label{E:3.41}
		\forall k\in\{K,\dots,  \ell  \wedge n \}\colon\quad |\Xi_k|\le k^\eta.
	\end{equation}
If no such number exists, we set $K:=\min\{\ell,\lfloor n/2\rfloor\} +1$. 
\end{definition}

Besides~$\eta$ and $\ell$, the control variable depends on the underlying domain as well as the point~$x_0$ in which the concentric decomposition is rooted but we will not mark that dependence explicitly. A key point is that the law of~$K$ can be controlled uniformly:

\begin{lemma}
	\label{lemma-3.9}
	Let~$\eta\in(0,1/4)$. For each ~$u\ge 1$, $a>0$, and~$\delta>0$, there is $k_0\ge 1$, and for each $\ell\ge k\ge k_0$ there is $N_0\ge 1$, such that for all $N\ge N_0$, $x_0\in D_N^\delta$  and  $t\in[0,a]$,
	\begin{equation}
		\label{e:lemma-3.9}
		\BbbP\bigl(K(\eta,\ell)>k\big|\,\Gamma_{N,u}(x_0,t)\bigr)<\delta.
	\end{equation}
\end{lemma}


Once the decorations have been dispensed with --- which will be achieved by restricting the value of the control variable and replacing the decorations by a bound --- the events under consideration will depend on the random walk only. It remains to do the same in the conditioning event. This is the content of:

\begin{lemma}
\label{lemma-3.10}
For $k=0,\ldots,n$ set
\begin{equation}\label{e:def-bk}
t_k:=\Var(\alpha S_k). 
\end{equation}
For all~$u\ge1$,~$a>0$ and $\delta>0$ there exist $v>0$ and~$c\in(0,\infty)$ and, for all $k_0\ge1$ there is~$N_0\ge1$ such that for all~$N\ge N_0$, $x_0\in D_N^\delta$ and $t\in[0,a]$,
\begin{equation}
\label{e:3-10}
\begin{aligned}
\BbbP\bigl( A\,\big|\,\Gamma_{N,u}&(x_0,t)\bigr)
\\
&\le c \BbbP\biggl( A\,\bigg|\,\bigcap_{k=0}^{n}\Bigl\{\alpha S_k\le \alpha  \frac{t_k}{t_n}m_N  +v\Bigr\}\cap\bigl\{ S_n= m_N-t\sqrt{\log N}\bigr\}\biggr)+\delta
\end{aligned}
\end{equation}
 holds  for all events~$A\in\sigma(S_k\colon 0\le k\le k_0)$.
\end{lemma}

The  proofs of Lemmas~\ref{lemma-3.9} and~\ref{lemma-3.10} are technical and somewhat unilluminating and so we defer them to Section~\ref{sec-5}. 

\subsection{ Random walk estimates}
In order to deal with  events arising in applications of the concentric decomposition,  we represent the random walk event via a standard  Brownian motion $\{W_t\colon t\ge0\}$. %
We use~$P^x$ to denote the law of~$W$ subject to~$P^x(W_0=x)=1$. We then have:

\begin{lemma}
\label{l:tk}
Let $a>0$ and $\delta>0$ and  recall the sequence $\{t_k\}_{k=0}^n$ from \eqref{e:def-bk}.  For each $N\ge 1$ and $t\in[0,a]$, there exists a coupling of
$\BbbP(h^{D_N}\in\cdot\,|\, -\alpha S_n + \alpha m_N = t\sqrt{\log N})$ to  $P^0(W\in\cdot|W_{t_n}=t\sqrt{\log N})$  such that  
\begin{equation}
\label{E:4.1u}
\biggl\{-\alpha S_k + \alpha \frac{ t_k  }{t_n}m_N\biggr\}_{k=0}^n \,\,=\,\, \{W_{t_k}\}_{k=0}^n.
\end{equation}
Moreover, there exists  a constant  $c\in(0,\infty)$ such that $|t_k - 4k|< c$  holds  for all $N\ge 1$,  all  $1\le k\le n$ and  all  $x_0\in D^\delta_N$.
\end{lemma}
\begin{proof}
The coupling follows  as $(S_k)_{k=0}^n$ is a Gaussian random walk by~\eqref{E:3.29}. Under the probability measure $\bbP$, the variance satisfies
\begin{equation}
	\Var\, S_k = \Var\,\bigl(\ol{(h^{D_N})_{\partial \Delta^k}}\bigr)(x_0) =  \Var\, h^{D_N}_{x_0} - \Var\, h^{\Delta^k}_{x_0}
\end{equation}
by~\eqref{e:harm-av} and the Gibbs-Markov property of the DGFF $h^{D_N}$. By the estimate for the Green function~\eqref{E:3.15a}, the r.h.s.\ in the last display equals
\begin{equation}
	g \log N - g \log(N\rme^{-k}) + O(1) = gk +O(1),
\end{equation}
where the $O(1)$ terms are bounded uniformly in $N\ge 1$, $1\le k\le N$ and $x_0\in D^\delta_N$. As $\alpha g =4$, this yields a coupling of
$\BbbP(h^{D_N}\in\cdot)$ to $P^0(W\in\cdot)$ such that
$ \{-\alpha S_k \}_{k=0}^n \,\,=\,\, \{W_{t_k}\}_{k=0}^n$. The assertion now follows by conditioning on $-\alpha S_n = W_{t_n}=- \alpha m_N + t\sqrt{\log N}$ and comparing the means.
\end{proof}

 As is seen from \eqref{E:4.1u}, the process~$\wh S$ defined by
\begin{equation}
\label{e:def-hat-Sk}
	\wh S_k := -\alpha S_k + \alpha \frac{ t_k}{t_n}m_N.
\end{equation} 
naturally couples to Brownian motion. 
 The process~$\wt S$ defined via
\begin{equation}\label{e:def-tilde-Sk}
	\wt S_k := -\alpha S_k + \alpha b_k \quad\text{for}\quad b_k := m_N - m_{N\rme^{-k}}.
\end{equation}
is in turn  well suited for extremal value theory for the DGFF. 
For the application of Lemma~\ref{l:tk}, we  will also need to  estimate the difference between these  processes: 
\begin{lemma}\label{l:diff-hat-tilde}
	Let $\delta>0$. There exists a constant $c\in(0,\infty)$ such that for all $N\ge 1$, $x_0\in D^\delta_N$, and all $0\le k \le n$,
	\begin{equation}\label{e:diff-hat-tilde}
		 -c\le \frac{t_k}{t_n}m_N - b_k  \le c\bigl(1+\log\min\{k,n-k\}\bigr)
\end{equation}
and hence
	\begin{equation}\label{e:diff-S-hat-tilde}
	-c\le \wh S_k - \wt S_k  \le c\bigl(1+\log\min\{k,n-k\}\bigr).
\end{equation}
\end{lemma}
\begin{proof}
We begin with the first inequality in~\eqref{e:diff-hat-tilde}.
For $k=0$ and $k=n$, the difference~$\frac{t_k}{t_n}m_N - b_k$ is bounded  uniformly in~$n$, so we will focus on $1\le k<n$. Abbreviate 
\begin{equation}
	e_s := -\log n + \log(n-s) + \frac{s}{n}\log n,\quad s\in(0,n).
\end{equation}
Then, by Lemma~\ref{l:tk} and the definitions of $b_k$ and $m_N$, it follows that
$b_k - \frac{t_k}{t_n}m_N - \frac{3}{4}\sqrt{g} e_k$ is bounded in absolute value by a constant that does not depend on~$k$,~$N$ and~$x_0$. Hence, it suffices to show that $e_s$ is bounded from above by a constant  uniformly in  $n\ge1$ and $s\in[1,n-1]$. 
To this aim, a derivative calculation
 shows that  that $s\mapsto e_s$ is convex. Since $e_1$ and $e_{n-1}$ are bounded, so is~$\max_{s\in[1,n-1]}e_s$.

For the second inequality in~\eqref{e:diff-hat-tilde}  we first note 
\begin{equation}
	\Big| \frac{t_k}{t_n}m_N - b_k  \Big| \le \Big| \frac{t_k}{t_n}m_N - m_k  \Big| + \big| m_k - b_k\big|,
\end{equation}
 We now invoke~\cite[Lemma~B.1]{Ballot}, the definitions of~$b_k$ and~$m_N$ and subadditivity of the logarithm to conclude that both terms on the right are bounded by a constant times $1+\log\min\{k,n-k\}$. This proves ~\eqref{e:diff-hat-tilde}; the  relation~\eqref{e:diff-S-hat-tilde} is a direct consequence of~\eqref{e:diff-hat-tilde} and the definitions of $\wh S_k$ and $\wt S_k$.
\end{proof}


 Next will give a bound on  the right-hand side in Lemma~\ref{lemma-3.10} in terms of a   $3$-dimen\-sio\-nal  Bessel process $\{Y_t\}_{t\ge 0}$; i.e., a solution to the SDE~$\textd Y_t=\textd W_t+Y_t^{-1}\textd t$. In order to keep track of the initial value, let~$\wt P^x$ denote the law of~$Y$ subject to~$\wt P^x(Y_0=x)=1$. 

\begin{lemma}
\label{l:bessel}
For all~$a>0$, 
$v\ge 0$ and $\delta>0$ there exist~$c\in(0,\infty)$ and~$\tilde v\ge 0$ and, for all $k_0\ge1$, there is~$N_0\ge1$ such that for all~$N\ge N_0$, $x_0\in D_N^\delta$ and~$t\in[0,a]$,
the following holds:
\begin{equation}
\begin{aligned}
\label{e:bessel-An}
P^0\biggl(\{W_{t_k}\}_{k=1}^{k_0}\in A\,\bigg|\,\bigcap_{k=0}^{n}\bigl\{W_{t_k}\ge -v\bigr\}\cap&
\bigl\{W_{t_n}=t\sqrt{\log N}\bigr\}\biggr)\\
&\leq c\wt P^{\,\tilde v}\biggl(\{Y_{t_k}-\tilde v\}_{k=1}^{k_0}\in A \biggr) + \delta
\end{aligned}
\end{equation}
for all measurable~$A\subseteq\bbR^{k_0}$. 
\end{lemma}

 The proof of Lemma~\ref{l:bessel} relies on similar arguments as those of Lemmas~\ref{lemma-3.9} and~\ref{lemma-3.10} and so we defer it to Section~\ref{sec-5} as well.  (Note that the above is a statement about Brownian bridges. The connection to~$x_0$ and~$D_N$ arises through the choice of~$n$ and the $t_k$'s.)  
Next we  note the following fact on stochastic domination for Bessel processes:

\begin{lemma}\label{l:Bessel-dom}
Let $v\ge 0$ and $b\ge 0$. Then for all  product measurable $A\subseteq \R^{[0,b]}$ such that~$1_A$ is non-decreasing in each coordinate,  
\begin{equation}
	P^v\bigl((Y_s-v)_{s\le b}\in A\bigr)
	\le P^0\bigl((Y_s)_{s\le b}\in A\bigr).
\end{equation}
\end{lemma}
\begin{proof}
We follow an argument from Lawler~\cite[p.19]{Lawler-Bessel}.  Let~$Y^x$ denote the solution to the SDE $\rmd Y^x_t= \rmd W_t + \frac1{Y^x_t}\textd t$ with initial value~$Y_0=x\ge0$, where~$W$ is a standard Brownian motion. (The transience of the 3-dimenional Bessel process then implies~$Y_t>0$ for all~$t>0$.) Using the same sample of~$W$ for~$Y^v$ and~$Y^0$, the difference $Y^v_t - Y^0_t$ then satisfies  the ODE
\begin{equation}
\textd (Y^v_t - Y^0_t) = -\frac{Y^0_t - Y^v_t}{Y^0_tY^v_t} \rmd t
\end{equation}
which gives us the representation 
\begin{equation}
	Y^v_t - Y^0_t = v\exp\left\{-\int_0^t \frac{\rmd s}{Y^0_s Y^v_s}\right\}.
\end{equation}
 The right-hand side is confined to $[0,v]$. Hence, $Y^v_t - v \le Y^0$ and $Y^v_t \ge Y^0_t$, which now implies the assertion. 
\end{proof}

\subsection{ Proof of Proposition \ref{prop-Z-sup} from Lemmas~\ref{lemma-3.9},~\ref{lemma-3.10} and~\ref{l:bessel}}
 We now move to the proof of  Proposition~\ref{prop-Z-sup}  conditional on the above technical lemmas. 
We will explicate all the necessary details only for the first inequality as the proof of the second inequality is very similar.

\begin{proof}[Proof of \eqref{E:2.3b} in Proposition~\ref{prop-Z-sup}]
Fix~$\epsilon>0$. 
We start by a reduction to limits along discrete sequences.  First note that the assumption that~$\psi$ is decreasing shows that, for any~$t>0$ and the unique natural~$j$ such that~$t\in[j,j+1)$,
\begin{equation}
\label{E:4.1}
\gamma\bigl(t\bigr)\ge\sqrt{1+j}\,\,\psi(j+1)\ge 2^{-1/2}\,\sqrt{2+j}\,\,\psi(j+1)=2^{-1/2}\,\gamma(j+1). 
\end{equation}
For any real~$r\le 1$ and the unique natural~$k$ such that $\texte^{-(k+1)}< r\le\texte^{-k}$ we then get 
\begin{equation}
\int_{\log(1/r)}^\infty\texte^{- 4\gamma(t)}\textd t
\le\sum_{j\ge k}\texte^{ -  2^{3/2} \gamma(j)}. 
\end{equation}
 Using also that $r\mapsto Z^D(B(\wh X,r))$ is non-decreasing we obtain 
\begin{equation}
\label{E:4.2}
\frac{Z^D\bigl(B(\wh X,r)\bigr)}{\int_{\log(1/r)}^\infty\texte^{- 4\gamma(t)}\textd t}
\ge \frac{Z^D\bigl(B(\wh X,\texte^{-(k+1)})\bigr)}{\sum_{j\ge k}\texte^{ -  2^{3/2} \gamma(j)}}.
\end{equation}
It follows that
\begin{equation}
\begin{aligned}
\label{E:4.3}
\quad
P\Bigl(\limsup_{r\downarrow0}&\frac{Z^D(B(\wh X,r))}{\int_{\log(1/r)}^\infty\rme^{-  4\gamma(t)}\textd t}\le 1\Bigr)
\\
&\le \lim_{\rho\downarrow0}\,\lim_{k_0\to\infty}\,\lim_{k_1\to\infty}
P\biggl(\,\max_{ k_0< k \le k_1}
\frac{Z^D\bigl(B^\infty(\wh X,\texte^{-k+1})\bigr)}{\sum_{j\ge k}\texte^{ -
2^{3/2}\gamma(j)}} \le1+\rho/2\biggr).
\end{aligned}
	\end{equation}
 By Lemma~\ref{lemma-3.1} (and with~$\Sigma$ as defined there), each interval~$[\texte^{-k-1},\texte^{-k}]$ contains a point $r_k\in\Sigma$ for which~$\{Z_N^D(B^\infty(\wh X_N,r_k))\}_{k=k_0}^{k_1}$ then tends in law to $\{Z^D(B^\infty(\wh X,r_k))\}_{k=k_0}^{k_1}$. A simple monotonicity argument then bounds probability on the right of \eqref{E:4.3}  by
\begin{equation}
\label{E:4.4}
\limsup_{N\to\infty}\BbbP
\biggl(\,\max_{ k_0\le k < k_1}\frac{Z^D_N\bigl(B^\infty(\wh X_N,\texte^{-k})\bigr)}
{\sum_{j\ge k}\texte^{ -2^{3/2}\gamma(j)}}\le 1+\rho\biggr).
\end{equation}
We can thus focus on the probability under the \textit{limes superior}.

Given integers $N\ge1$ and $1\le k_0<k_1$ and parameters $b,\rho>0$, let $ E_N = E_N (k_0,k_1,b,\rho)$ denote the intersection of $\{Z^D_N(D)\ge\texte^{-b}\}$ with the event in \eqref{E:4.4}. Following the strategy described in Section~\ref{sec-3.1} and denoting 
\begin{equation}
\begin{aligned}
\label{E:4.5a}
q_N=q_N(u,a,b,&\delta):=
\BbbP(\max h^{D_N}> m_N+u)+\BbbP\bigl(Z^D_N(D)<\texte^{-b}\bigr)
\\
&+\BbbP(N\wh X_N\not\in D_N^\delta)+\BbbP\Bigl(h^{D_N}(N\wh X_N)-m_N\not\in[-a\sqrt{\log N},0]\Bigr),
\end{aligned}
	\end{equation}
the probability in \eqref{E:4.4} is at most
\begin{equation}
\label{E:4.5}
q_N+\sum_{x\in D_N^\delta}\,
\BbbP\Bigl(\wh X_N=x/N,\,0\le m_N-h_x^{D_N}\le a\sqrt{\log N},\,h^{D_N}\le m_N+u\,;\, E_N \Bigr).
\end{equation}
By the tail bound \eqref{E:3.18}, Lemmas~\ref{lemma-3.1}--\ref{lemma-3.2} in conjunction with $Z^D(D)>0$ and $Z^D(\partial D)=0$ a.s.\ ensure that $q_N(u,a,b,\delta)$ vanishes in the limit $N\to\infty$ followed by $u\to\infty$, $a\to\infty$, $b\to\infty$ and~$\delta\downarrow0$ (in any order). In particular, there are choices of $u$, $a$, $b$ and~$\delta$ such that~$q_N<\epsilon/2$ for all sufficiently large~$N$.

With $u$, $a$, $b$ and~$\delta$ fixed, we use Lemmas~\ref{lemma-3.3} and \ref{lemma-3.4} to bound the probability under the sum in \eqref{E:4.5} by
\begin{equation}
\label{E:4.6}
\frac{\texte^{O(1)}}{N^2}\texte^b\,\sup_{0\le t\le a}\max_{x\in D_N^\delta}
\BbbP
\biggl(\,\max_{k_0\le k \le k_1} \frac{Z^D_N\bigl(B^\infty(x/N,\texte^{-k}\bigr)}
{\sum_{j\ge k}\texte^{ -2^{3/2}\gamma(j)}} \le1+\rho\,\bigg|\,\Gamma_{N,u}(x,t)\biggr),
\end{equation}
where~$\texte^b$ arises from the bound on~$Z^D_N(D)$ enforced by event~$ E_N $ being applied in the expectation in \eqref{E:3.9}. For the probability on the right we invoke the concentric decomposition rooted at~$x$. Write~$\{S_k\}$ for the random walk \eqref{E:3.29} and~$K=K(\eta,\ell)$ for the control variable~\eqref{E:3.41} defined for an exponent~$\eta$ such that (as assumed in the statement) $t^{-\eta}\gamma(t)\to\infty$. Note that
\begin{equation}
\alpha(m_N-m_{N\texte^{-k}})=4k+o(1),\qquad \alpha b_k = 4k + o(1)
\end{equation}
as $N\to\infty$. The $o(1)$ terms (that vanish as~$N\to\infty$ for each~$k$) are in absolute value at most~$\rho/2$ uniformly in~$k$ satisfying $k_0\le k\le k_1$ once~$N$ is sufficiently large.   Assuming that~$n\ge\ell$ and recalling  the definition of~$\wt S_k$ from \eqref{e:def-tilde-Sk}, using 
\eqref{E:3.39} and \eqref{E:3.41}  while noting  $Z^D_N\bigl(B^\infty(x,\texte^{-k})\bigr) \ge \sum_{j= k-\kappa(\delta)}^{\ell} Z^D_N\bigl(A_j\bigr)$, we get
\begin{equation}
k\ge K\quad\Rightarrow\quad
Z^D_N\bigl(B^\infty(x,\texte^{-k})\bigr)\ge\sum_{j= k-\kappa(\delta)}^{\ell-\kappa(\delta)}\texte^{-\wt S_{j}-j^\eta-\rho}.
\end{equation}
Hence,  once  $k$ is large enough  that $\wt S_k - \wh S_k \le \rho$, which is possible by Lemma~\ref{l:diff-hat-tilde},  for large enough $N$ and~$\ell$ (and hence $n$ large enough), the probability in \eqref{E:4.6} is bounded~by 
\begin{equation}
	\label{E:4.11}
	\BbbP\Bigl(K>k_0\,\Big|\,\Gamma_{N,u}(x,t)\Bigr)+\BbbP\biggl(\,\bigcap_{k=k_0-\kappa(\delta)}^{k_1-\kappa(\delta)}\bigl\{\wh S_{k}\ge  2^{3/2} \gamma(k)-k^\eta-3\rho\bigr\}\,\bigg|\,\Gamma_{N,u}(x,t)\biggr).
\end{equation}
By Lemma~\ref{lemma-3.9}, the control variable is tight so 
the first term in \eqref{E:4.11} tends to zero as $N\to\infty$, $\ell\to\infty$ and~$k_0\to\infty$ uniformly in~$x\in D_N^\delta$ and $t\in[0,a]$.
Lemma~\ref{lemma-3.10} in turn shows that the probability on the right is bounded  by a constant times 
\begin{equation}
\label{E:4.20u}
\BbbP\biggl(\,\bigcap_{k=k_0-\kappa(\delta)}^{k_1-\kappa(\delta)}\bigl\{\wh S_{k}\ge2^{3/2}\gamma(k)-k^\eta-  3 \rho\bigr\}\,\bigg|\,\bigcap_{k=0}^{n}\bigl\{\wh S_{k}\ge -v\bigr\}, \wh S_{n}=\alpha t\sqrt{\log N}\biggr)
\end{equation}
for ~$v$ as in Lemma~\ref{lemma-3.10}, uniformly in all parameters.

Following the steps described in Subsection~\ref{sec-3.3}, we now realize the random walk $\wt S$ as the standard Brownian motion evaluated at times $t_0<t_1<\dots<t_n$ where  $t_k:=\Var(\wh S_{k})$.  Lemma~\ref{l:tk} gives $t_k\ge k$ and $t_k \le  4^{4/3}k$ for $k_0\le k\le n-k_0$ and sufficiently large $k_0$. Then, analogously to~\eqref{E:4.1}, we have
\begin{equation}\label{e:4.14}
2^{3/2}\gamma(k) = 2^{3/2}\sqrt{1+k}\,\psi(k)
\ge 4^{1/12}\sqrt{1+4^{4/3}k}\,\psi\bigl(t_k\bigr)\ge  2^{1/6}\gamma\bigl(t_k\bigr)
\end{equation}
for sufficiently large $k$.
Using  the bound from  \eqref{e:diff-hat-tilde} and  the domination by Bessel process~$Y$ from  Lemma~\ref{l:bessel}, the probability on the right of~\eqref{E:4.11} is no larger than
\begin{equation}
\label{E:4.22i}
c\wt P^{\tilde v}\biggl(\,\bigcap_{k=k_0-\kappa(\delta)}^{k_1-\kappa(\delta)}\bigl\{ Y_{t_k} - \tilde v\ge  2^{1/6}\gamma(t_k)-k^\eta- 3\rho\bigr\}\biggr)+\ep'
\end{equation}
for~$k_0$ sufficiently large and for~$\tilde v$ as in Lemma~\ref{l:bessel}.  Here~$\ep'$ is~$\epsilon/2$ divided by the multiplicative  constants that arose  in \eqref{E:4.6} and on the way to \eqref{E:4.20u}. 

 It remains to compute the $k_1,k_0\to\infty$ limit of the probability in \eqref{E:4.22i}. First we use Lemma~\ref{l:Bessel-dom} to effectively change the starting point of~$Y$ to zero which bounds the probability in~\eqref{E:4.22i} by 
\begin{equation}
	\label{E:4.13}
	\wt P^0\biggl(\,\bigcap_{k=k_0-\kappa(\delta)}^{k_1-\kappa(\delta)}\bigl\{ Y_{t_k} \ge  2^{1/6}\gamma(t_k)-k^\eta- 3\rho\bigr\} \biggr).
\end{equation}
 Next pick~$u\in(1,2^{1/6})$ and observe that $( 2^{1/6}-u)\gamma(t_k)\ge  2(k^\eta+ 3\rho)$ for sufficiently large~$k$  (recall that we assumed~$\eta\in(0,1/4)$). Note also that  the assumed downward monotonicity of $s\mapsto \psi(s):=\sqrt{1+s}\,\gamma(s)$ implies
\begin{equation}
\gamma(s)=\sqrt{1+s}\,\,\psi(s)\le\sqrt{1+t_k}\,\psi(t_{k-1})=\sqrt{\frac{1+t_k}{1+t_{k-1}}}\,\,\gamma(t_{k-1}),\quad s\in[t_{k-1},t_k].
\end{equation}
By Lemma~\ref{l:tk} again, the term on the extreme right is less than~$u\gamma(t_{k-1})$ once~$k$ is large enough. 
If~$Y_{t_{k-1}}\ge 2^{1/6}\gamma(t_{k-1})-k^\eta- 3\rho$, we either have~$Y_s\ge\gamma(s)$ for all~$s\in[t_{k-1},t_k]$ or there exists~$s\in[t_{k-1},t_k]$ such that
\begin{equation}
Y_s-Y_{t_{k-1}}<\gamma(s)-[2^{1/6}\gamma(t_{k-1})-k^\eta- 3\rho]\le -\frac12(2^{1/6}-u)\gamma(t_{k-1}).
\end{equation}
Introducing 
\begin{equation}
\Psi'_k:=\min_{s\in [t_{k-1},t_k]}( Y_{s} - Y_{t_{k-1}}),
\end{equation}
the probability in~\eqref{E:4.13} is thus at most
\begin{equation}
\label{E:4.26i}
P\biggl(\,\bigcap_{s\in[t_{k_0},t_{k_1}]} \bigl\{ Y_s \ge
\gamma(s) \bigr\}\biggr)
+\sum_{k=k_0-\kappa(\delta)}^{k_1-\kappa(\delta)}P\biggl( \Psi'_k \le -\frac12(2^{1/6}-u)\gamma(t_k) \biggr)
\end{equation}
 by way of a union bound.

 In order to control the expression \eqref{E:4.26i} recall that the Bessel process~$Y$ obeys the SDE $\textd Y_t = \textd W_t+Y_t^{-1}\textd t$. The drift term is positive (for~$Y_t>0$) and so a downward excursion of the Bessel process is less likely than that of the Brownian motion. Since $|t_k-t_{k-1}|$ is (by Lemma~\ref{l:tk}) bounded uniformly in~$k$ and~$n$, it follows that 
\begin{equation}
\label{e:DD'}
P(\Psi'_k<-\lambda) \le  c \,\rme^{- c'\lambda^2/2},
\end{equation}
 holds for all $\lambda\ge 0$ with some constants $c,c'\in(0,\infty)$ independent of~$k$ and~$n$. 
 In the limit as $k_1\to\infty$ followed by $k_0\to\infty$, the first term in \eqref{E:4.26i} is then bounded by the probability on the right of \eqref{E:2.3b} while, by \eqref{e:DD'} and the assumed polynomial growth of~$\gamma$, the second term in \eqref{E:4.26i} tends to zero. Combining this with all the preceeding calculations, we get~\eqref{E:2.3b} as desired. 
\end{proof}

Here are the changes needed in the proof of the second inequality:

\begin{proofsect}{Proof of \eqref{E:2.4} in Proposition \ref{prop-Z-sup}}
As already noted, the proof is quite analogous to that of~\eqref{E:2.3b}. First, for  $t\in[j,j+1]$ with natural~$j\ge1$, the downward monotonicity of~$\psi$ shows 
\begin{equation}
	\gamma(t) =\sqrt{1+t}\,\,\psi(t)\le \sqrt{2+j-1}\,\,\psi(j-1) \le 2  \sqrt{1+j-1}\,\, \psi(j-1) = 2\gamma(j-1)
\end{equation}
and so for~$k$ related to~$r<1/\texte$ via $\texte^{-(k+1)}\le r\le\texte^{-k}$,
\begin{equation}
\frac{Z^D(B(\wh X,r))}{\int_{\log(1/r)}^\infty\rme^{-\gamma(t)/2}\textd t} \le
\frac{Z^D\bigl(B^\infty(\wh X,\texte^{-k})\bigr)}{\sum_{j\ge k}\rme^{-\gamma(j-1)}}.
\end{equation}
 Abbreviate
\begin{equation}
A_{k,\ell}^\infty(x):=B^\infty(x,\texte^{-k})\smallsetminus B^\infty(x, \texte^{-\ell}).
\end{equation}
 Using Lemma~\ref{lemma-3.1}  along with the fact that $Z^D$ does not charge singletons (which translates into $Z^D(B^\infty(\wh X,\texte^{-\ell}))\to0$ in probability as~$\ell\to\infty$),  the probability on the left of \eqref{E:2.4} is thus at most
\begin{equation}
\lim_{\rho\downarrow0}\,\lim_{k_0\to\infty}\,\lim_{k_1\to\infty}
\lim_{\ell\to\infty}\,\limsup_{N\to\infty}
\BbbP
\biggl(\,\max_{k_0\le k \le k_1}\frac{Z^D_N\bigl( A_{k,\ell}^\infty(\wh X_N)\bigr)}{\sum_{j\ge k}\rme^{-\gamma(j)}}\ge1-\rho\biggr).
\end{equation}
Following the argument \twoeqref{E:4.5}{E:4.6}, the probability on the right is bounded by
\begin{equation}
\label{E:4.31i}
q_N+\texte^{O(1)+b}\,\sup_{0\le t\le a}\max_{x\in D_N^\delta}
\BbbP
\biggl(\,\max_{k_0\le k \le k_1}\frac{Z^D_N\bigl( A_{k,\ell}^\infty(x/N)\bigr)}{\sum_{j\ge k}\rme^{-\gamma(j)}}
\ge1-\rho\,\bigg|\,\Gamma_{N,u}(x,t)\biggr),
\end{equation}
where~$q_N$ is as in \eqref{E:4.5a}.
We now choose~$u$, $a$, $b$ and~$\delta$ so that~$\limsup_{N\to\infty}q_N<\epsilon$ and,   using the notation from Section~\ref{sec-3.3}, observe that
\begin{equation}
	Z^D_N\bigl( A_{k,\ell}^\infty(x_0/N)\bigr) = \sum_{j=k}^{\ell-1} Z^D_N(A_j).
\end{equation}
Recalling \eqref{E:3.39}, if $\wt S_j>\gamma(j)+j^\eta-\log(1-\rho)$ and~$ \Xi_j  \le j^\eta$ for all~$j$ in the range of the above sum, then the left hand side is less than~$(1-\rho)\sum_{j\ge k}\texte^{-\gamma(j)}$. It follows that, for~$\rho>0$ so small that $-\log(1-\rho)\le 2\rho$, the probability on the right of \eqref{E:4.31i} is bounded~by 
\begin{equation}
\label{e:ctr-lb-wt}
\BbbP\Bigl(K>k_0\,\Big|\, \Gamma_{N,u}(x,t) \Bigr)+
\BbbP\Bigl(\,\bigcup_{k=k_0}^{k_1+\ell}\bigl\{
\wt S_k\le\gamma(k)+k^\eta+2\rho
\bigr\}\,\Big|\, \Gamma_{N,u}(x,t) \Bigr).
\end{equation}
 It remains to estimate this expression for small~$\rho$ and~$N\gg\ell\gg k_1\gg k_0$. 
 From Lemma~\ref{l:diff-hat-tilde}, it follows that $\wh S_k - \wt S_k \le (2\eta - 1)k^\eta$ once $k$ is sufficiently large. Hence, the expression in~\eqref{e:ctr-lb-wt} is bounded from above by
\begin{equation}
	\label{e:ctr-lb}
	\BbbP\Bigl(K>k_0\,\Big|\, \Gamma_{N,u}(x,t) \Bigr)+
	\BbbP\Bigl(\,\bigcup_{k=k_0}^{k_1+\ell}\bigl\{
	\wh S_k\le\gamma(k)+(2k)^\eta+2\rho
	\bigr\}\,\Big|\, \Gamma_{N,u}(x,t) \Bigr).
\end{equation}

From Lemma~\ref{l:tk} and an argument similar to~\eqref{e:4.14}, we obtain $\gamma(k)\leq 4\gamma(t_k/8)$  as soon as~$k$ is  sufficiently large.
We also recall that $t_k\ge  2k$ for all sufficiently large $k$. Applying also Lemma~\ref{lemma-3.10},  the bound  \eqref{e:diff-hat-tilde} and Lemma~\ref{l:bessel}, we then dominate the second probability in~\eqref{e:ctr-lb} by
\begin{equation}
c\wt P^{\tilde v}\biggl(\,\bigcup_{k=k_0}^{k_1+\ell}\bigl\{ Y_{t_k}\le4\gamma(t_k/8)+t_k^\eta+2\rho\bigr\}\biggr) + \ep,
\end{equation}
where $\tilde v$ is as in Lemma~\ref{l:bessel}.
By Lemma~\ref{l:Bessel-dom}, the probability is further bounded by
\begin{equation}
	\label{E:4.23p}
	\wt P^0\biggl(\,\bigcup_{k=k_0}^{k_1+\ell}\bigl\{ Y_{t_k} \le4\gamma(t_k/8)+t_k^\eta+2\rho+ \tilde v\bigr\}\biggr)
\end{equation}
 which is bounded in  the limit as $\ell\to\infty$, $k_1\to\infty$ followed by $k_0\to\infty$ and $\rho\downarrow 0$ by
\begin{equation}\label{e:p:2.4-sc}
P\biggl(\liminf_{s\to\infty}\,\bigl[ Y_{s} - 4\gamma(s/8)-s^\eta-\tilde v\bigr]\le 0\biggr).
\end{equation}
 As the $3$-dimensional Bessel process is the modulus of the $3$-dimensional Brownian motion,  the Brownian scaling shows that $(Y_s)_{s\ge 0}$ is equidistributed to $(2\sqrt{2}Y_{s/8})_{s\ge 0}$.
Moreover, $2(1-\tfrac{1}{\sqrt{2}})\gamma(s/8)>\tfrac{1}{\sqrt{2}}s^\eta + \tfrac{\tilde v}{2\sqrt{2}}$ for sufficiently large~$s$. Hence, the probability in~\eqref{e:p:2.4-sc} is bounded by the probability on the right-hand side of~\eqref{E:2.4}.
\end{proofsect}

\section{Proofs of technical lemmas}
\label{sec-5}
\noindent  
Here we collect the  technical  estimates based on the concentric decomposition and prove Lemmas~\ref{lemma-3.9}-\ref{lemma-3.10} and~\ref{l:bessel}.

\subsection{ Proof of Lemma~\ref{lemma-3.10}}
Recall the setup of the concentric decomposition from Section~\ref{sec:conc}.
We start with two auxiliary lemmas that allow us to swap the conditioning on $\Gamma_{N,k}(x_0,t)$ from \eqref{E:3.20a} for a conditioning that only involves the random walk $\{S_k\}_{k=1}^n$.
 Invoking the notation  $ \wh S_k$ from~\eqref{e:def-hat-Sk},  
 for all $v\ge 0$ define the event 
\begin{equation}
\label{E:5.1w}
	 F_N (v):= \bigl\{  \wh S_n  =\alpha t\sqrt{\log N}\bigr\}\cap  \bigcap_{k=0}^{n} \bigl\{  \wh S_k  \ge -v \bigr\},
\end{equation}
 where~$\alpha:=2/\sqrt g$ and~$n$ is related to~$N$ and  the  implicit parameters~$\delta>0$ and~$x_0\in D_N^\delta$ via \eqref{E:3.16}. The following lemmas effectively say that, modulo multiplicative constants, the event $\Gamma_{N,k}(x_0,t)$  in the conditioning can be replaced with~$F_N(v)$ once~$v$   is sufficiently~large:

\begin{lemma}\label{lemma-A-Gamma}
	 For all $\epsilon>0$, $u\ge 1$, $a>0$ and $\delta>0$ there exists $v_0>0$ such that
\begin{equation}
\sup_{N\ge1}\sup_{x_0\in D^\delta_N}\,\sup_{t\in[0,a]}\,\sup_{v\ge v_0}\bbP\bigl(  F_N (v)^\cc\,\big|\, \Gamma_{N,u}(x_0,t)\bigr)\le\epsilon.
\end{equation}
\end{lemma}

\begin{lemma}\label{lemma-Gamma-A}
	 For all $u\ge 1$, $a>0$, $\delta>0$ and $v>0$ there exists $c\in(1,\infty)$ such that 
\begin{equation}
\inf_{ N\ge 3}\,\inf_{x_0\in D^\delta_N}\,\inf_{t\in[0,a]}\, \frac{\bbP\bigl(\Gamma_{N,u}(x_0,t)\,\big|\, h_{x_0} = m_N -t\sqrt{\log N} \bigr)}{\bbP\bigl( F_N  (v) \,\big|\, h_{x_0} = m_N -t\sqrt{\log N} \bigr) }  \ge c^{-1}.
\end{equation}
\end{lemma}

 Both lemmas can be proved  based on what we already know. We start with: 

\begin{proofsect}{Proof of Lemma~\ref{lemma-A-Gamma}}
	 By Lemma~\ref{l:diff-hat-tilde}, there exists a constant $c\in(0,\infty)$ such that
	\begin{equation}
			\bbP\Big(\wh S_k < -v \,\Big|\, \Gamma_{N,u}(x_0,t) \Big) \le 	\bbP\Big(\wt S_k < -v -c \,\Big|\, \Gamma_{N,u}(x_0,t) \Big).
	\end{equation}
In the following, we write $\wt v:= v + c$. 
 Explicating the definition~\eqref{E:3.20a} of $\Gamma_{N,u}(x_0,t)$, the union bound dominates the conditional probability in the statement by the sum of 
		\begin{equation}
		\label{e:lemma-A-Gamma-p3}
			\bbP\Big(\wt S_k < -\wt v \,\Big|\, \Gamma_{N,u}(x_0,t) \Big)
			=\frac{\bbP\Big(\wt S_k < -\wt v \,, h^{D_N}\leq m_N + u \,\Big|\, 
				h^{D_N}_{x_0}  = m_N  - t\sqrt{\log N} \Big)}{\bbP\Big( h^{D_N}\leq m_N + u \,\Big|\, 
				h^{D_N}_{x_0}  = m_N - t\sqrt{\log N} \Big) }
		\end{equation}
 over~$k= 1,\dots,n-1 $. To include $k=0,n$ in~\eqref{E:5.1w}, we use that $\wt S_0=0$ and that
\begin{equation}
	\wt S_n= -\alpha h^{D_N}_{x_0} + \alpha m_N=\alpha t\sqrt{\log N} -\alpha m_{N\rme^{-n}}\ge -v
\end{equation}
for sufficiently large $v$ by~\eqref{E:3.16}. 
	Our aim is to show that the sum  over~\eqref{e:lemma-A-Gamma-p3}  vanishes as~$\wt v\to\infty$, uniformly in $N\ge 1$, $x_0\in D^\delta_N$ and $t\in[0,a]$. 
	
	
 First observe that the definition of $\wt S_k$ along with~\eqref{e:def-bk} give 
	\begin{equation}\label{e:Sb}
			\wt S_k	= -\alpha S_k + \alpha b_k = -\alpha S_k + \alpha \left( m_N - m_{N\rme^{-k}}\right).
	\end{equation}
 This implies 
	\begin{equation}
		\{\wt S_k < -\wt v\}=\Bigl\{S_k - m_N + m_{N\rme^{-k}} > \tfrac{\wt v}{\alpha} \Bigr\}.
	\end{equation}
 Next we will invoke the representation~\eqref{e:harm-av} of~$S_k$ as the harmonic average of the values of~$h^{D_N}$ on~$\partial\Delta^k$. Along with a shift of $h^{D_N}$ by $-m_N-u$,  this gives 
\begin{equation}
\begin{aligned}
\bbP\Bigl(\wt S_k < -\wt v& \,, h^{D_N}\leq m_N + u \,\Big|\, 
		h^{D_N}_{x_0}  = m_N  - t\sqrt{\log N} \Bigr)\\
		&\le \bbP\biggl(\ol{(h^{D_N})_{\partial \Delta^k}}(x_0) + m_{N\rme^{-k}} + u > 
		\tfrac{\wt v}{\alpha}  \,, h^{D_N}\leq 0 \\
		&\qquad\qquad\qquad\qquad\bigg|\, 
		h^{D_N}_{x_0}  = - u - t\sqrt{\log N}, (h^{D_N})_{\partial D_N} = -m_N - u \biggr).
\end{aligned}
\end{equation}
 where $(h^{D_N})_{\partial D_N}$ refers to the set of values of~$h^{D_N}$ on~$\partial D_N$.  Denoting
\begin{equation}
\WW_k(u):=\bigl\{w\in\R^{\partial\Delta^k}\colon \ol{w_{\partial \Delta^k}}(x_0)\ge \tfrac{\wt v}{\alpha} - u\bigr\}
\end{equation}
the Gibbs-Markov property applied to conditioning on the values on~$\partial\Delta_{k+1}$ then allows us to bound  the probability on the right-hand side from above by
\begin{equation}
\begin{aligned}
\label{e:lemma-A-Gamma-p15}
		&\int_{w\in\WW_k(u)} \bbP\Big((h^{D_N})_{\partial \Delta^k} + m_{N\rme^{-k}} \in \rmd w \,\Big|\,
		h^{D_N}_{x_0}  = - u - t\sqrt{\log N}, (h^{D_N})_{\partial D_N}= -m_N - u\Big)
		\\&\qquad\times \bbP\Big(\max_{x\in \Delta^{k+1}\smallsetminus\{x_0\}} h^{D_N}_x \le 0\,\Big|\, (h^{D_N})_{\partial \Delta^k} + m_{N\rme^{-k}}  =  w, h^{D_N}_{x_0}= - u -t\sqrt{\log N}\Big)
		\\&\qquad\qquad\times \bbP\Big(\max_{x\in D_N\smallsetminus \Delta^{k-1}} h^{D_N}_x \le 0\,\Big|\, (h^{D_N})_{\partial D_N}= -m_N - u, (h^{D_N})_{\partial \Delta^{k}} + m_{N\rme^{-k}} = w\Big)\,.
\end{aligned}
\end{equation}
 This expression can now be estimated using \cite[Theorem~1.8]{Ballot}. Indeed, there exists a constant~$c\in(0,\infty)$ such that for the second term in \eqref{e:lemma-A-Gamma-p15} we get 
\begin{equation}
\begin{aligned}
\bbP\Big(\max_{x\in \Delta^{k+1}\setminus\{x_0\}} h^{D_N}_x \le 0\,\Big|\, (h^{D_N})_{\partial \Delta^k}  +& m_{N\rme^{-k}} =  w, h^{D_N}_{x_0}= - u - t\sqrt{\log N}\Big)
\\
&\le c\frac{(u + t\sqrt{\log N})\rme^{-\, \ol{w_{\partial \Delta^k}}( x_0)^+ + 2 j}}{n-k}
\end{aligned}
\end{equation}

 whenever~$j$ satisfies
\begin{equation}
	j\ge \max_{x,y\in  \Delta^{k+1} }
	\bigl| \ol{w_{\partial \Delta^k}}(x)- \ol{w_{\partial \Delta^k}}(y)\bigr|.
\end{equation}
 To see this, we apply~\cite[Theorem 1.8]{Ballot} with outer domain $U:=B^\infty(0,\rme^{-\kappa(\delta)})$, discretized by~\eqref{e:discretization}, and inner domain $V:= B(0,1/4)$, discretized by~\cite[Eq.~(1.3)]{Ballot}. Then, the outer domain scaled by $N\rme^{-k}$ equals $\Delta^k - x_0$ by~\eqref{E:3.17}, and $V^-_0=\{0\}^\cc$. 
Analogously, for the term in the third line of~\eqref{e:lemma-A-Gamma-p15}, we have
\begin{equation}
\begin{aligned}
\bbP\Big(\max_{x\in D_N\setminus\Delta^{k-1}} h^{D_N}_x \le 0\,\Big|\, (h^{D_N})_{\partial D_N}= &-m_N - u, (h^{D_N})_{\partial \Delta^k} + m_{N\rme^{-k}} = w\Big)
	\\
	&\le c\frac{u\rme^{- \,\ol{w_{\partial \Delta^k}}( x_0)^+ + 3 j}}{k}
\end{aligned}
\end{equation}
 whenever~$j$ obeys 
\begin{equation}\label{e:osc-j}
	j\ge \max_{x,y\in  \bbZ^2 \setminus \Delta^{k-1}  \cup \{x_0\}  }\bigl| \ol{w_{\partial \Delta^k}}(x) - \ol{w_{\partial \Delta^k}}(y)\bigr|.
\end{equation}
Above, we also used that $\ol{w_{\partial \Delta^k}}(\infty)  - \ol{w_{\partial \Delta^k}}(x_0) \le j$.  
Since the right-hand side in these estimates depends only on the value of~$\ol{w_{\partial \Delta^k}}(x_0)$,  plugging them into~\eqref{e:lemma-A-Gamma-p15}  and invoking elementary disintegration  bounds the expression in~\eqref{e:lemma-A-Gamma-p15} by a constant times
\begin{equation}
\label{e:lemma-A-Gamma-p2}
		\int_{[\tfrac{\wt v}{\alpha}-u,\infty)}  \mu_{N,u,t}(\textd r)\sum_{j\ge 1} q_{N,u,t}(j) \frac{u(u+t\sqrt{\log N})}{k(n-k)}\rme^{-2 r + 5 j},
\end{equation}
where~$\mu_{N,u,t}$ is the probability measure
\begin{equation}
\mu_{N,u,t}(\cdot):= 
\bbP\Big(\ol{(h^{D_N})_{\partial \Delta^k}}(x_0) + m_{N\rme^{-k}} \in \cdot \,\Big|\,
		h^{D_N}_{x_0}  = - u - t\sqrt{\log N}, h^{D_N}_{\partial D_N}= -m_N - u\Big)
\end{equation}
and~$q_{N,u,t}(j)$ is given by
\begin{equation}
\begin{aligned}
\label{E:5.16i}
 q_{N,u,t}(&j):= 
\bbP\Big( \max_{x,y\in (  \bbZ^2  \setminus \Delta^{k-1}) \cup \Delta^{k+1}}\Bigl|\ol{(h^{D_N})_{\partial \Delta^k}}(x) - \ol{(h^{D_N})_{\partial \Delta^k}}(y)\Bigr| \ge j-1\\ &\,\Big|\, \ol{(h^{D_N})_{\partial \Delta^k}}(x_0) + m_{N\rme^{-k}} = r, h^{D_N}_{x_0}  = - u - t\sqrt{\log N}, (h^{D_N})_{\partial D_N}= -m_N - u\Big).
\end{aligned}
\end{equation}
 We will now estimate the terms entering \eqref{e:lemma-A-Gamma-p2} separately. 

 Since $\ol{(h^{D_N})_{\partial \Delta^k}}(x_0) + m_{N\rme^{-k}}$ is Gaussian with variance $g (k\wedge(n-k)) + O(1)$ and mean $-u -\tfrac{k}{n}t\sqrt{\log N} + O(\log(1+k\wedge(n-k)))$ by~\cite[Lemma~B.1, Propositions~B.2 and~B.3]{Ballot}, for the measure we get 
	\begin{equation}\label{e:est-mu}
		\mu_{N,u,t}(\textd r)\le c' (k\wedge(n-k))^{-1/2}\exp\Bigl\{-\rho\frac{r + \frac{k}{n}t\sqrt{\log N}}{k\wedge(n-k)}\Bigr\}\,\textd r
	\end{equation}
  for any $\rho>0$ and   a ($\rho$-dependent) constant~$c'>0$ uniformly on $[\tfrac{\wt v}{\alpha}-u,\infty)$. As to the term in \eqref{E:5.16i}, here \cite[Proposition B.4]{Ballot} gives 
	\begin{equation}\label{e:est-q}
	q_{N,u,t}(j)\le c''\rme^{- 5   j +  \lambda  \frac{|r|}{k\wedge (n-k)} + \lambda \frac{t\sqrt{\log N}}{n-k}}
	\end{equation}
	for  some constants~$c'',\lambda>0$. 
	We now plug~\eqref{e:est-mu} and~\eqref{e:est-q} into~\eqref{e:lemma-A-Gamma-p2}, and we choose the quantity $\rho\in(\lambda,\infty)$ so large that
	\begin{equation}
		\rho\frac{\tfrac{k}{n} t\sqrt{\log N}}{k\wedge(n-k)} \ge \lambda  \frac{t\sqrt{\log N}} {n-k},\quad k=1,\ldots,n-1.
	\end{equation}
Combining these  above estimates and assuming that~$v$ is so large  that $\wt v/\alpha -u \ge 0$, we bound the expression in~\eqref{e:lemma-A-Gamma-p2}  and thus the numerator on the right of \eqref{e:lemma-A-Gamma-p3}  by a constant~times
	\begin{equation}
		\frac{1}{n}(k\wedge (n-k))^{-3/2} u(u+t\sqrt{\log N})\rme^{-2 \tfrac{\wt v}{\alpha} + 2 u}
	\end{equation}
	 for $k=1,\ldots,n-1$.  In the same vein, \cite[Corollary~1.7]{Ballot} readily shows  that the denominator on the right-hand side of~\eqref{e:lemma-A-Gamma-p3} is bounded from below by a constant times $u(u+t\sqrt{\log N})/n$. Hence, the conditional probability on the left-hand side of~\eqref{e:lemma-A-Gamma-p3} is bounded by a constant times
	$(k\wedge (n-k))^{-3/2} \rme^{-2 \wt v/\alpha + 2 u}$.  The assertion follows by summing this  over~$k=1,\dots,n-1$  and taking~$\wt v$ sufficiently large.
\end{proofsect}

 The proof of the second auxiliary lemma is considerably shorter:

\begin{proofsect}{Proof of Lemma~\ref{lemma-Gamma-A}}
By definition of $\Gamma_{N,u}(x_0,t)$ and Lemma~\ref{lemma-3.4},
\begin{equation}
\label{E:5.23}
\begin{aligned}
	\bbP\bigl(\Gamma_{N,u}(x_0,t)\,\big|\, &h_{x_0} = m_N -t\sqrt{\log N} \bigr)\\
	&= \bbP\bigl(h^{D_N} \le m_N + u \,\big|\, h_{x_0} = m_N -t\sqrt{\log N} \bigr)
	\ge c'\frac{1+t}{\sqrt{\log N}}
\end{aligned}
\end{equation}
for a constant $c'>0$.

By~\eqref{E:3.24} and~\eqref{e:def-hat-Sk},
\begin{equation}\label{e:p-FN}
	\bbP\bigl( F_N  (v) \,\big|\, h_{x_0} = m_N -t\sqrt{\log N} \bigr) = 
	\bbP\bigl( \min_{k \le n} \wh S_k  \ge -v \,\big|\, \wh S_n = \alpha t \sqrt{\log N} \bigr).
\end{equation}
By Lemma~\ref{l:tk}, $\wh S$ is an inhomogeneous random walk bridge  with initial and terminal values  $\wh S_0=0$ and $\wh S_n =  \alpha t\sqrt{\log N}$ and step variances $t_k=\Var\bigl(\wh S_k  -\wh S_{k-1}\bigr)$  for which $|t_k - 4k|$ is bounded uniformly in~$0\le k\le n$. Hence,   \cite[Lemma~2.3]{CHLsupp}  
bounds the right-hand side of~\eqref{e:p-FN} by a constant times
\begin{equation}
\frac{1+t\sqrt{\log N}}{\log N}.
\end{equation}
(The result~\cite{CHLsupp} is applied to  the Brownian motion from Lemma~\ref{l:tk} by setting $-\wh S_k=W_{t_k}$.)  In combination with \eqref{E:5.23}, this gives the desired claim. 
\end{proofsect}

 With the above lemmas established, we can give: 

\begin{proofsect}{Proof of Lemma~\ref{lemma-3.10}}
	For each $v\ge 0$,  a union bound shows that  the probability on the left-hand side of~\eqref{e:3-10} is at most
	\begin{equation}\label{e:3.10-Bayes}
		\bbP\bigl(  F_N (v)^\cc\,\big|\, \Gamma_{N,u}(x_0,t)\bigr)+
		\bbP\bigl( A \,\big|\,  F_N (v)\bigr) \frac{\bbP\bigl( F_N (v)\,\big|\, h^{D_N}_{x_0} = m_N - t\sqrt{\log N}\bigr)}{\bbP\bigl(\Gamma_{N,u}(x_0,t)\,\big|\, h^{D_N}_{x_0} = m_N - t\sqrt{\log N}\bigr)}.
	\end{equation}
	The claim thus follows from Lemmas~\ref{lemma-A-Gamma} and~\ref{lemma-Gamma-A}.
\end{proofsect}

\subsection{ Proof of Lemma~\ref{lemma-3.9}}
\noindent
 We start with estimates on the various objects entering the ``decorations'' part of the concentric decomposition.

\begin{lemma}
\label{l:bchi}
	For each $\delta>0$, $a\ge 0$,  $v \ge 0$  and $u\ge 1$ there exists $c\in(0,\infty)$ such that
	\begin{equation}\label{e:b}
	\bbP\biggl(\, \max_{y\in\Delta^{k}\smallsetminus\Delta^{k+1}}\sum_{j=0}^k
	\bigl|\mathfrak b_j(y)\bigr|\bigl|\varphi_j(x_0)\bigr| > \lambda 
	\,\bigg|\, \min_{ i\le n}\wh S_i\geq -v\,,
	\wh S_{n} = \alpha t\sqrt{\log N}\biggr)
	\le c^{-1}\rme^{-c\lambda^2}
	\end{equation}
and
	\begin{equation}\label{e:chi}
	\bbP\biggl( \,\max_{y\in\Delta^k\smallsetminus\Delta^{k+1}} \sum_{j=0}^{k-2}\bigl|\chi_j(y)\bigr| > \lambda \biggr)
	\le c^{-1}\rme^{-c\lambda^2}
\end{equation}
for all $\lambda\ge 0$, $N\ge 1$, $x_0\in D^\delta_N$ and $t\in[0,a]$.
\end{lemma}

For the proof of Lemma~\ref{l:bchi}, we consider the random walk increments
\begin{equation}\label{e:phiS}
	\varphi_j(x_0) = S_j - S_{j-1}= \alpha^{-1} (\wh S_j - \wh S_{j-1}) - \alpha^{-1}\frac{t_{j} - t_{j-1}}{t_n} m_N
\end{equation}
where the equalities follow from~\eqref{E:3.29} and~\eqref{e:Sb}, and where $\frac{t_{j} - t_{j-1}}{t_n} m_N$ is bounded by a constant by~\eqref{E:3.3}  and Lemma~\ref{l:tk}.  We need the following tail bound:

\begin{lemma}\label{l:crw-inc}
	For each $\delta>0$, $a\ge 0$,  $v \ge 0$ , there exists $c\in(0,\infty)$ such that
	\begin{equation}\label{e:crw-inc}
		\BbbP\biggl( \bigl| \wh S_k - \wh S_{k-1} \bigr|\ge \lambda \,\bigg| \bigcap_{j=1}^{n} \{\wh S_j \ge -v\}, \wh S_n = \alpha t\sqrt{\log N}\biggr)\le c^{-1}\rme^{-c \lambda^2}
	\end{equation}
	for all $\lambda\ge 0$, $N\ge 1$, $x_0\in D^\delta_N$, $t\in[0,a]$ and $1\le k\le n$.
\end{lemma}
\begin{proof}
	We write the left-hand side in~\eqref{e:crw-inc} as
	\begin{equation}\label{e:p1:crw-inc}
		\frac{\BbbP\Bigl( \bigl| \wh S_k - \wh S_{k-1} \bigr|\ge \lambda,
			\bigcap_{j=1}^{n}\{\wh S_j \ge -v\}\,\Big| \wh S_n = \alpha t\sqrt{\log N}\Bigr)}{
			\BbbP\Bigl(		\bigcap_{j=1}^{n} \{\wh S_j \ge -v\}\,\Big| \wh S_n =\alpha t\sqrt{\log N}\Bigr)}.
	\end{equation}
	 By Lemma~\ref{l:tk}, $(\wh S_k)$ is a random walk with step variances $t_k = \Var(\wh S_k - \wh S_{k-1})$  such that $t_k - 4k$ is bounded uniformly in $1\le k\le n$.  By the Markov property at $k-1$ and $k$, the numerator equals
	\begin{equation}
	\begin{aligned}
	\label{e:p2:crw-inc}
		&\int_{y,y'\ge -v: |y-y'|\ge \lambda}
		\BbbP\Bigl( \wh S_{k-1} \in \rmd y, \wh S_k\in\rmd y' \,\Big|\, \wh S_n =\alpha t\sqrt{\log N}\Bigr)\\
		&\quad\times \BbbP\Bigl( \bigcap_{j= 1}^{k-1}  \{\wh S_j \ge -v\}\,\Big|\, \wh S_{k-1}=y\Bigr)
		\BbbP\Bigl( \bigcap_{j=k+1}^{n  }  \{\wh S_j \ge -v\}\,\Big|\, \wh S_{k}=y', \wh S_n=\alpha t\sqrt{\log N}\Bigr).
	\end{aligned}
	\end{equation}
	Assuming that $k\le n/2$, we bound the first factor in the second line by a  constant times $k^{-1}(1+v)(1+v+y)$,  by e.g.~\cite[Lemma~2.3]{CHLsupp} and the second by a constant times $(1+v+y')(1+v+ t\sqrt{\log N}) /(n-k)$. 
	Plugging these estimates into~\eqref{e:p2:crw-inc}, 	changing to the variable $z:=y'-y$, and  using the bound $1+a+b \leq (1+a)(1+b)$ for $a,b > 0$, we bound the numerator in~\eqref{e:p1:crw-inc} by a constant times
	\begin{equation}
	\begin{split}\label{e:p3:crw-inc}
		& \frac{(1+v)^4  \big(1+t\sqrt{\log N}\big)}{kn}
		 \int_{y\ge -v} \BbbP\Bigl( \wh S_{k-1} \in \rmd y\,\Big|\, \wh S_n =0\Bigr) (1+|y|)^2 \\
		& \qquad \qquad \qquad \int_{|z|\ge \lambda}\BbbP\Bigl( \wh S_{k} - \wh S_{k-1} \in \rmd z \,\Big|\, \wh S_{k-1} = y,\wh S_n = \alpha t\sqrt{\log N}\Bigr) (1+|z|) \,.
	\end{split}
	\end{equation} 
	By Lemma~\ref{l:tk} and  and the definition of $\wh S_k$, it follows that 
	\begin{equation}
\BbbP\Bigl( \wh S_{k} - \wh S_{k-1} \in \rmd z \,\Big|\, \wh S_{k-1} = y,\wh S_n = \alpha t\sqrt{\log N}\Bigr)
\end{equation}
 is a Gaussian distribution with variance bounded from above and below, and mean bounded in absolute value by a constant times $|y|/(n-k)+1$, where we used~\eqref{E:3.16}. Hence, the inner integral in~\eqref{e:p3:crw-inc} is bounded by a constant times
 \begin{equation}
\rme^{-c\lambda^2 + cy^2/n^2} \,,
\end{equation}
for some $c\in(0,\infty)$. Furthermore, 
\begin{equation}
\BbbP\Bigl( \wh S_{k-1} \in \rmd y\,\Big|\, \wh S_n = \alpha t\sqrt{\log N}\Bigr)
\end{equation}
 is a Gaussian distribution with variance bounded by a constant times $k$ by Lemma~\ref{l:tk}, and mean bounded by a  constant times $\frac{k}{n}t \sqrt{\log N}$. It follows that the double integral in~\eqref{e:p3:crw-inc} is at most a constant times
\begin{equation}
\Big(k + \tfrac{k}{n}\sqrt{\log N}\big)^2\Big) \rme^{-c\lambda^2}
 \leq c^{-1} k \rme^{-c\lambda^2} \,,
\end{equation}
for sufficiently small $c > 0$. Altogether the numerator in~\eqref{e:p1:crw-inc} is thus bounded by a constant times
\begin{equation}
\frac{	(1+v)^4 (1+t\sqrt{\log n})}{n} \,.
\end{equation} 

By Lemma~\ref{l:tk} and monotonicity, the denominator in~\eqref{e:p1:crw-inc} is bounded from below by
\begin{equation}
	\BbbP\Bigl(		\min_{s\in[0,t_n]} W_s \ge -v \,\Big| W_{t_n} =\alpha t\sqrt{\log N}\Bigr).
\end{equation}
This probability is bounded from below by a constant times $(1+v)(1+v+t\sqrt{\log N})/n$ by e.g.~\cite[Proposition~2.1]{CHLsupp}.  This yields the assertion for $k\le n/2$. The case $k\ge n/2$ follows by symmetry, we then bound the second factor in~\eqref{e:p2:crw-inc} by $1$ and use~\cite[Lemma~2.3]{CHLsupp} for the first factor.
\end{proof}

 We are now ready to give:

\begin{proofsect}{Proof of Lemma~\ref{l:bchi}}
By a union bound, the probability in~\eqref{e:b} is bounded by
\begin{equation}
	\label{e:p-incr1}
	\sum_{j=0}^k \bbP\Big(
	4\rme^{(k-j)/2}\max_{y\in\Delta^{k}\smallsetminus\Delta^{k+1}} \big|\mathfrak b_j(y) \big| \big|\varphi_j(x_0)\big| >\lambda
	\,\Big|\, \min_{i\le n}\wh S_i\geq -v\,,
	\wh S_{n} = \alpha t\sqrt{\log N}\Big).
\end{equation}
As in~\cite[Lemmas~3.6 and~3.7]{BL3}, we have
\begin{equation}
	\rme^{k-j}\max_{y\in\Delta^{k}\smallsetminus\Delta^{k+1}}\big|\mathfrak b_j(y) \big|\leq c'
\end{equation}
for a constant $c'\in(0,\infty)$ and all $j\leq k \leq n$. Thus we bound
the probability in~\eqref{e:p-incr1} from above by
\begin{equation}
	\bbP\Big( |\varphi_j(x_0)| > c'^{-1}\rme^{(k-j)/2}\lambda /4
	\,\Big|\, \min_{ i\le n}\wh S_i\geq -v\,,
	\wh S_{n} = \alpha t\sqrt{\log N}\Big).
\end{equation}
This is further bounded by $c\rme^{-c\lambda^2}$ for a constant $c\in(0,\infty)$ by~\eqref{e:phiS} and Lemma~\ref{l:crw-inc}.

The probability in~\eqref{e:chi} is bounded by
\begin{equation}
		\sum_{j=0}^{k-2} \bbP\Big( \rme^{(k-j)/2}\max_{y\in\Delta^k\smallsetminus\Delta^{k+1}} \bigl|\chi_j(y)\bigr| > \lambda \Big)
\end{equation}
by a union bound. As in~\cite[Lemma~3.8]{BL3}, we have
\begin{equation}
	\bbP\Big( \rme^{(k-j)/2} \max_{y\in\Delta^k\smallsetminus\Delta^{k+1}}\bigl|\chi_j(y)\bigr| > \lambda \Big)
	\leq\text{const.}\rme^{-\text{const.}\rme^{(k-j)/2}\,\lambda^2}
\end{equation}
and the assertion follows by summing over $j$.
\end{proofsect}

 With Lemma~\ref{l:bchi} in hand, we are ready for: 

\begin{proofsect}{Proof of Lemma~\ref{lemma-3.9}}
Let $u\ge 1$, $a>0$, $\delta>0$. We use the decomposition~\eqref{e:3.10-Bayes} with the  choice  $ A:= \{K(\eta,\ell)>k\}$ and then invoke Lemmas~\ref{lemma-A-Gamma} and~\ref{lemma-Gamma-A} to  note that, for~$v$ large enough, the first term on the right is small and the ratio of the probabilities is bounded. It then  suffices to show for some sufficiently large~$v$  there exists $k_0\ge 1$, and for each $\ell\ge k\ge k_0$ there exists $N_0\ge 1$, such that
\begin{equation}\label{e:claim-p-3.9}
	\bbP\bigl( K(\eta,\ell)>k \mid  F_N (v) \bigr) < \delta
\end{equation}
for all $N\ge N_0$, $x_0\in D^\delta_N$, $t\in[0,a]$.

We use~\eqref{e:Xi-k} and union bounds in
\begin{equation}
\begin{aligned}
	\label{e:Xi-k-UB}
	\BbbP\Bigl(\Xi_k>&k^\eta/2 \,\Big|\,  F_N (v) \Bigr)\\
	&\le \BbbP\biggl( \frac{1}{n} \sum_{x\in\Delta^k\smallsetminus\Delta^{k+1}}
	\exp\Big\{\alpha\big(h'_k(x) +\chi_k(x) +  \chi_{k-1}(x) -m_{N\rme^{-k}}\big)\Big\}>\rme^{k^\eta/4}\biggr)\\
	&\qquad\qquad+  \BbbP\biggl(\alpha \sum_{j=0}^k \max_{x\in\Delta^k\smallsetminus\Delta^{k+1}}\bigl|\frakb_j(x)\bigr|\bigl|\varphi_j(x_0)\bigr|> k^\eta / 8 \,\bigg|\,  F_N (v)\biggr)\\
	&\qquad\qquad\qquad\qquad+ \BbbP\biggl(\,\alpha\sum_{j=0}^{k-2} \max_{x\in\Delta^k\smallsetminus\Delta^{k+1}}\bigl|\chi_j(x)\bigr|> k^\eta / 8 \biggr).
\end{aligned}
\end{equation}
Here we left out the conditioning in the first term and the third term on the right-hand side using the independence of $(S_k)$ from $h'_k$ and $\chi_j$.
Using again a union bound, we  dominate  the first term on the right-hand side by
\begin{equation}
\begin{aligned}
	\label{e:Xi-k-UB-1}
	\BbbP\Biggl( \frac{1}{n}& \sum_{x\in\Delta^k\smallsetminus\Delta^{k+1}}
	\exp\Big\{\alpha\big(h'_k(x) +\chi_k(x) +  \chi_{k-1}(x)\\
	&\quad+ (1+\frakb_k(x))\varphi_k(x_0)
	+ (1+\frakb_{k-1}(x))\varphi_{k-1}(x_0) - m_{N\rme^{-k}}\big)\Big\}>\rme^{k^\eta/8}\Biggr)\\
	&+\BbbP\biggl( \alpha \max_{x\in\Delta^k\smallsetminus \Delta^{k+1}} \bigl[(1+\frakb_k(x))\bigl|\varphi_k(x_0)\bigr|
	+ (1+\frakb_{k-1}(x))\bigl|\varphi_{k-1}(x_0)\bigr|\bigr] >  k^\eta/8 \biggr).
\end{aligned}
\end{equation}
As $|\frakb_k|$ is bounded by a constant $c<\infty$ as in~\cite[Lemma~3.7]{BL3}, the second probability in~\eqref{e:Xi-k-UB-1} is bounded by
\begin{equation}
	\BbbP\Bigl( \alpha |\varphi_k(x_0)| >(1+c)^{-1} k^\eta/16 \Bigr)
	+\BbbP\Bigl( \alpha |\varphi_{k-1}(x_0)| >(1+c)^{-1} k^\eta/16 \Bigr),
\end{equation}
which in turn is bounded by a constant times $\rme^{-\text{const. } k^{2\eta}}$ by~\eqref{E:3.29}, \eqref{e:phiS} and Lemma~\ref{l:tk}. Furthermore, by Lemma~\ref{l:concdec}, the first term in~\eqref{e:Xi-k-UB-1} equals
\begin{equation}\label{e:Xi-k-UB-e}
	\BbbP\biggl( \frac{1}{n} \sum_{x\in\Delta^k\smallsetminus\Delta^{k+1}}
	\exp\Big\{\alpha \bigl(h^{\Delta^{k  -  1}}_x -m_{N\rme^{-k}}\bigr)\Big\}>\rme^{k^\eta/8}\biggr),
\end{equation}
which, by a union bound and the Markov inequality, is bounded by
\begin{equation}
	\begin{aligned}
	\label{e:UB-Z-cont}
	\BbbP\Bigl(\max  &\,h^{\Delta^{k-1}} > m_{N\rme^{-k}} + u' \Bigr)\\
	&+\rme^{-k^\eta/8}\bbE\biggl(\frac{1}{n} \sum_{x\in\Delta^k\smallsetminus\Delta^{k+1}}
	\exp\Big\{\alpha \bigl(h^{\Delta^{k-1}}_x -m_{N\rme^{-k}}\bigr)\Big\} \,;\,
	h^{D_N} \leq m_N + u' \biggr)
\end{aligned}
\end{equation}
uniformly in $k$ and $u'\ge 1$.

By~\cite[Proposition~5.1]{BGL}, the expectation in the last display is bounded by a constant times $u'$, and the probability on the left is bounded by a constant times $u'\rme^{-\alpha u'}$ by Lemma~\ref{lemma-DGFF-tail}.
Choosing $u'=k^\eta/8$, we can thus bound the expression in~\eqref{e:UB-Z-cont}, and hence also the first term on the right-hand side of~\eqref{e:Xi-k-UB},
by a constant times  $\rme^{-k^\eta/9}$.

The second and third terms on the right-hand side of~\eqref{e:Xi-k-UB} are bounded by a constant times $\rme^{- c k^\eta}$ by Lemma~\ref{l:bchi}  for some constant~$c>0$. 
It follows that the left-hand side of~\eqref{e:Xi-k-UB} is bounded by a constant times $\rme^{- c k^\eta}$.

Now we come to a stochastic lower bound for $\Xi_k$. Analogously to~\twoeqref{e:Xi-k-UB}{e:Xi-k-UB-e}, we have
\begin{equation}
	\begin{aligned}\label{e:p:ctrl-lb}
	\BbbP\Bigl( -\Xi_k > &\, k^\eta /2\,\Big|\,  F_N (v) \Bigr)
	\\
	&\leq \BbbP\biggl( \frac{1}{n} \sum_{x\in\Delta^k\smallsetminus\Delta^{k+1}}
	\exp\Big\{\alpha \bigl(h^{\Delta^{k  -  1}}_x -m_{N\rme^{-k}}\bigr)\Big\}
	 \leq 
	\rme^{-k^\eta/8}\biggr)
	+  c^{-1}\rme^{- c  k^\eta}.
\end{aligned}
\end{equation} 
 for some constant~$c>0$. 
By~\cite[Theorem~2.3]{BGL} and as we may assume that $n-k\ge n/2$ (so that we may replace $\tfrac1n$ with $\tfrac{1}{n-k}$ in~\eqref{e:p:ctrl-lb} at the cost of a factor $2$),
the probability on the right-hand side is  at most $\rme^{-ck^\eta}$ plus 
\begin{equation}
P\Bigl(Z^B( A )
 \leq 
2\rme^{-k^\eta/8}\Bigr)
\end{equation}
 for fixed~$k$ and all~$N$ large enough ,  where~$B=[-1/2,1/2]^2$ and~$A:=\texte^{-1.1}B\smallsetminus \texte^{-1.9}B$  and where we  used the definition~\eqref{E:3.17} of $\Delta^k$.  Above we take a slightly smaller annulus then the one corresponding to the scaling limit of the domain in the last sum so that to accommodate an annulus in between that is also a stochastic continuity set for the random measure~$Z$. 

The last probability in turn is bounded by
\begin{equation}
	2\rme^{-k^\eta/8}E\Bigl((Z^B(A)^{-1}\Bigr)
\end{equation}
by the Markov inequality. The expectation is finite by~\cite[Corollary~14]{DRSV2}. Hence, there exists $k_0\in\N$ and for each $k\ge k_0$ there exists $N_0\in \N$ such that the l.h.s.\ of~\eqref{e:p:ctrl-lb} is bounded by  constant times $\rme^{-ck^\eta}$  for all $N\ge N_0$.

From all the above, we then obtain
\begin{equation}
	\BbbP\biggl( \, \bigcup_{j=k}^\ell\bigl\{\bigl|\Xi_j\bigr| > j^\eta\bigr\}\,\bigg|\, \Gamma_{N,u}(x_0,t) \biggr)\leq \delta/2
\end{equation}
for all sufficiently large $k$ and $N$ (depending on $k$), as required.
\end{proofsect}

\subsection{Proof of Lemma~\ref{l:bessel}}
\label{sec-5.4}
\noindent 
Our last item of business is a proof of Lemma~\ref{l:bessel} for which we need: 

\begin{lemma}
	\label{l:Doob}
	Let  $W$ under $P^0$  be a standard Brownian motion, and let $b, v\ge0$. Then, for each  measurable $A\subseteq \R^{[0,b]}$,
	\begin{equation}
		\lim_{T\to\infty}  P^0\Bigl(( W_s)_{s\le b}\in A\,\Big|\, \min_{s\le T}  W_s\geq -v,  W_{T}=\alpha t\sqrt{\log N}\Bigr)= 
		\wt P^v\Bigl((Y_s -v)_{s\le b}\in A\Bigr),
	\end{equation}
	 where~$Y$ under~$\wt P^v$ is the three-dimensional Bessel process started at $Y_0=v$. 
\end{lemma}
\begin{proof}
The proof of~\cite[Lemma~7]{Harris} passes through despite the $n$-dependent boundary value at  $W_T$.  In conjunction with the Portmanteau theorem,  this yields the claim. 
\end{proof}

 Next we will need a way to restrict a minimal growth of the random walk  $k\mapsto W_{t_k}$, where $W$ under $P^0$ is a standard Brownian motion, and the $t_k$ are from Lemma~\ref{l:tk}.   For this we introduce another control variable defined by
\begin{equation}\label{e:tau}
	\tau:=\max\bigl\{k=1,\ldots,\lfloor n/2\rfloor\colon
	\min\{  W_{t_k} \,,  W_{t_{n-k}}  \}\leq k^\delta \bigr\}
\end{equation}
 with the proviso $\tau:=0$ if the set under the maximum is empty. 
As it turns out, $\tau$ is tight  under  the relevant law: 

\begin{lemma}
	\label{l:tau}
	For each $\ep>0$,  $\delta\in(0,1/8)$,  $a\ge 0$ and $v\ge 0$ there exists $M<\infty$ such that
	\begin{equation}\label{e:l:tau}
		\sup_{N\ge 1}\,\sup_{x_0\in D^\delta_N}\,\sup_{t\in[0,a]}\,P^0\Bigl(\tau > M
		\,\Big|\, \min_{k \le n}  W_{t_k}  \geq -v\,,
		 W_{t_{n}}  = \alpha t\sqrt{\log N} \Bigr)<\ep.
	\end{equation}
\end{lemma}

\begin{proof}
	 We rewrite the conditional probability in~\eqref{e:l:tau} as
	\begin{equation}\label{e:rep-RW}
		\frac{P^0\Big( \tau > M\,,
			\min_{k\le n} W_{t_k} \geq -v
			\,\Big|\, W_{t_n} = \alpha t\sqrt{\log N} \Big)}{
			P^0\Big(\min_{k \le n} W_{t_k} \geq -v
			\,\Big|\, W_{t_n} = \alpha t\sqrt{\log N} \Big)}.
	\end{equation} 
By Lemma~\ref{l:tk},   the sequence $\{t_k\}_{k=1}^n$ 
is such that~$t_k-4k$ is bounded uniformly in $1\le k\le n$. 
 In~\eqref{e:rep-RW}, the denominator is bounded from below by
\begin{equation}\label{e:denom-rep-RW}
	P^0\Bigl(\min_{s\in[0,t_n]} W_s \geq -v
	\,\Big|\, W_{t_n} = \alpha t\sqrt{\log N} \Bigr)
\end{equation}
 which is  bounded from below by a constant times
	$\sqrt{\log N}/n$
	 by~\cite[Proposition~2.1]{CHLsupp}. We claim that the numerator is bounded from above by a constant times
	$M^{-1/8}\sqrt{\log N}/n$.
	Once this claim is proved, the expression in~\eqref{e:rep-RW} is bounded by a constant times $M^{-1/8}$  which implies the tightness of~$\tau$. 
	
	 The numerator in~\eqref{e:rep-RW} equals
	\begin{equation}
		P^0\Big( \min_{M+1\le k\le \lfloor n/2\rfloor } \min\{W_{t_k}, W_{t_{n-k}}\} \le k^\delta\,,
		\min_{j\le n} W_{t_j} \geq -v
		\,\Big|\, W_{t_n} = \alpha t\sqrt{\log N} \Big)
	\end{equation}
	by~\eqref{e:tau}, which by symmetry is equal to
	\begin{equation}
		P^0\Big( \max_{M+1\le k\le \lfloor n/2\rfloor } \max\{W_{t_k}, W_{t_{n-k}}\} \ge - k^\delta\,,
		\max_{j\le n} W_{t_j} -v \leq 0
		\,\Big|\, W_{t_n} = -\alpha t\sqrt{\log N} \Big).
	\end{equation}
	 Using the union bound,  the probability is bounded from above by
	\begin{equation}
	\begin{aligned}\label{e:p-rep-UB}
		\sum_{k=M+1}^{\lfloor n/2\rfloor}P^0\Big( W_{t_k} \ge &- k^\delta,\,
		\max_{j\le n} W_{t_j} -v \leq 0
		\,\Big|\, W_{t_n} = -\alpha t\sqrt{\log N} \Big)\\
		&+\sum_{k=M+1}^{\lfloor n/2\rfloor}P^0\Big( W_{t_{n-k}} \ge - k^\delta\,,
		\max_{j\le n} W_{t_j} -v \leq 0
		\,\Big|\, W_{t_n} = -\alpha t\sqrt{\log N} \Big)
	\end{aligned}
	\end{equation}
	 By the Markov property, the term corresponding to~$k$ in the first sum equals 
	\begin{equation}
	\label{E:5.63}
	\begin{aligned}
		\int_{w\in [-k^\delta, v]}  \rho_k(\textd w)& 
		 P^0\Big( 		\max_{j\le k} W_{t_j} -v \leq 0
		\,\Big|\, W_{t_k} = w \Big)
		\\
		&\times P^0\Big( 		\max_{k\le j\le n} W_{t_j} -v \leq 0
		\,\Big|\, W_{t_k} = w , W_{t_n}=-\alpha t\sqrt{\log N}\Big),
	\end{aligned}
	\end{equation}
	 where~$\rho_k$ is the Borel measure defined by
	\begin{equation}
\rho_k(A):=P^0\Big( W_{t_k}\in \,A\,\Big|\, W_{t_n} = -\alpha t\sqrt{\log N} \Big)
\end{equation}
	By Lemma~\ref{l:tk} and standard facts about Gaussian random variables,		
	\begin{equation}
		\rho_k(\textd w)  
		\le \text{const.\ } k^{-1/2} \rmd w.
		\end{equation}
	 By~\cite[Lemma~2.3]{CHLsupp}, the first probability in the integrand in \eqref{E:5.63} is bounded by a constant times $(1+v)(1+ v-w )/k$ , and the probability in the second line is bounded by a constant times $(1+ v-w  )(1+v+\alpha t\sqrt{\log N})/(n-k)$.
	 Using these facts in \eqref{E:5.63}  bounds the first sum in~\eqref{e:p-rep-UB} by a constant times
	\begin{equation}
	\begin{aligned}\label{e:p-b-rep}
		 \frac{(1+v)(1+v+\alpha t\sqrt{\log N})}{n} &\sum_{k=M+1}^{\lfloor n/2\rfloor}(v+k^{1/8}) (1+v+k^{1/8})^2 k^{-3/2}\\
		  &\le \text{const.} \frac{(1+v)^4(1+v+\alpha t\sqrt{\log N})}{n} M^{-1/8}.
	\end{aligned}
	\end{equation}
	The second sum in~\eqref{e:p-rep-UB} is bounded in the same way by~\eqref{e:p-b-rep}.    This bounds the numerator in~\eqref{e:rep-RW} by a constant times $M^{-1/8}\sqrt{\log N}/n$ as claimed. 
\end{proof}

\begin{proofsect}{Proof of Lemma~\ref{l:bessel}}
	By a union bound,  for any $\tilde v \geq v$ ,
	\begin{equation}
	\label{e:dec-bessel-p1}
	\begin{aligned}
		P^0\biggl(&\{W_{t_k}\}_{k=1}^{k_0}\in A\,\bigg|\,\bigcap_{k=0}^{n} \bigl\{W_{t_k}\ge - v\bigr\},
		W_{t_n}=\alpha t\sqrt{\log N}\biggr)\\
		&\qquad\le P^0\biggl( \min_{s\le t_n} W_s< - \tilde v \,\bigg|\,\bigcap_{k=0}^{n} \bigl\{W_{t_k}\ge -v\bigr\},
		W_{t_n}=\alpha t\sqrt{\log N}\biggr)\\
		&\qquad\qquad\qquad+ P^0\Bigl(\{W_{t_k}\}_{k=1}^{k_0}\in A \,\Big|\min_{s\le t_n} W_s\geq -\tilde v, W_{t_n}=\alpha t\sqrt{\log N}\Bigr)\\
		&\qquad\qquad\qquad\qquad\qquad\qquad\times\frac{P^0\bigl(\,\min_{s\le t_n} W_s\geq -\tilde v\mid W_{t_n}=\alpha t\sqrt{\log N}\bigr)}{
		P^0\bigl(\,\bigcap_{k=0}^{n} \bigl\{W_{t_k}\ge - v\bigr\}\,\big|\,
		W_{t_n}=\alpha t\sqrt{\log N}\bigr)}.
	\end{aligned}
	\end{equation}
Now it suffices to show that the first term on the right-hand side is bounded by $\delta$ for sufficiently large $\tilde v$, and that the ratio in the last line is bounded by a constant~$c$  for any such $\tilde v$ uniformly in~$n$. 

For the latter statement, we note that the ratio is bounded by
\begin{equation}\label{e:p:Bessel-ratio}
	\frac{P^0\bigl(\min_{s\le t_n} W_s\geq -\tilde v\mid W_{t_n}=\alpha t\sqrt{\log N}\bigr)}{
		P^0\bigl( \min_{s\le t_n} W_s \ge -v\,\big|\,
		W_{t_n}=\alpha t\sqrt{\log N}\bigr)}.
	\end{equation}
 Using the Ballot Theorem for Brownian motion, the numerator is equal to
\begin{equation}
(2+o(1))(\tilde{v} + \alpha t \sqrt{\log N})\frac{\tilde{v}}{t_n}
\end{equation}
and the denominator is equal to the same expression only with~$v$ in place of $\tilde v$. In particular, the ratio is bounded from above, for any choice of~$\tilde v$ and~$t$. 
	
 Turning to the first term on the right hand side of~\eqref{e:dec-bessel-p1}, 
for each~$k=0,\dots,n-1$ set $\wt\Psi_k:=\{\wt \Psi_k(s)\colon s\in[0,1]\}$  with  
\begin{equation}
	\wt \Psi_{k}(s):=\frac{1}{\sqrt{t_{k+1}-t_k}} \bigl(W_{st_{k+1}+(1-s)t_{k}}- W_{t_k}
	-s(W_{t_{k+1}} - W_{t_k}) \bigr),
\end{equation}
 and observe that $\{\wt \Psi_k\}_{k=0}^{n-1}$ is a family of i.i.d.\ Brownian bridges from~$0$ to~$0$ that are independent of $\{W_{t_k}\}$. 
Then denote 
\begin{equation}
	\label{e:Psi}
	\Psi_k:=\sqrt{t_{k+1}-t_k}\min_{s\in[0,1]}  \wt \Psi_{k}(s) \,,
\end{equation}

and observe that if $W_{t_k}, W_{t_{k+1}} \geq k^\delta$, 
\begin{equation}
	W_{st_{k+1}+(1-s)t_{k}} =  (1-s)W_{t_k} + sW_{t_{k+1}} + \Psi_k(s)
	\geq k^{\delta} + c \wt \Psi_k(s) \,,
\end{equation}
for some $c < \infty$, where we have also used Lemma~\ref{l:tk}.

Recalling~\eqref{e:tau}, we can then use the union bound to bound  the first term on the right-hand side of~\eqref{e:dec-bessel-p1} by 
\begin{equation}
\begin{aligned}	
P^0\biggl( \min_{s\le t_n} &W_s< - \tilde v \,\bigg|\,\bigcap_{k=0}^{n}\bigl\{W_{t_k}\ge -v\bigr\},
	W_{t_n}=\alpha t\sqrt{\log N}\biggr)\\
	&\le P^0\biggl( \tau > M  \,\bigg|\, \bigcap_{k=0}^{n} \bigl\{W_{t_k}\ge -v \bigr\}, W_{t_n}=\alpha t\sqrt{\log N}\biggl)
	+ \sum_{k=M}^{n-M} P^0\bigl( \Psi_k <  -c^{-1}  k^\delta \bigr)\\
	&\qquad+ P^0\biggl( \min_{s\in[0,t_M]\cup [t_{n-M},t_n]} W_s  < -\tilde v\,\bigg|\, \bigcap_{k=0}^{n} \bigl\{W_{t_k}\ge -v \bigr\}, W_{t_n}=\alpha t\sqrt{\log N}\biggr),
\end{aligned}
\end{equation}
 where we left out the conditioning in the second term on the right-hand side\ by independence of $\{W_{t_k}\}$ and $\{\Psi_k\}$. 

 Now given $\delta > 0$ , we first choose $M$ sufficiently large  so  that, using Lemma~\ref{l:tau} and that $\Psi_k$ has Gaussian tails, the terms in the second line add up to at most  $\delta/2$ . Then we choose~$\tilde v$ sufficiently large that the probability in the last line is  also  bounded by  $\delta/2$. For the latter, we can use the FKG property for Brownian motion to replace the conditional event by $W_{t_n} = 0$, and then the fact that the minimum in the last line has Gaussian tails under the law 
$P^0(\cdot\mid W_{t_n}=0)$, with constants that depend only on $M$. The  desired claim  now follows from~\eqref{e:dec-bessel-p1} and  the limit claim in  Lemma~\ref{l:Doob}.
\end{proofsect}

\section*{Acknowledgments}
\nopagebreak\nopagebreak\noindent
This project has been supported in part by the NSF award DMS-1954343, ISF grants No.~1382/17 
and No.~2870/21 and BSF award 2018330.
The work of the second author started when he was a  postdoctoral  fellow at the Technion.
The second author has also been supported in part by a Zeff Fellowship at the Technion, by a Minerva Fellowship of the Minerva  Stiftung  Gesellschaft fuer die Forschung mbH, and by DFG (German Research Foundation) grant  Ki 2337/1-1, 2337/1-2  (project 432176920).

\bibliographystyle{abbrv}

\end{document}